\documentclass[12pt,a4paper]{article}
\usepackage{amsmath,amssymb,amsthm,graphicx,color}
\usepackage[left=1in,right=1in,top=1in,bottom=1in]{geometry}
\usepackage{cite}
\usepackage[british]{babel}

\newtheorem{theo}{Theorem}[section]
\newtheorem{lm}{Lemma}[section]

\newtheorem{cor}{Corollary}[section]
\newtheorem{rmk}{Remark}[section]
\newtheorem{rmks}{Remarks}[section]

\newtheorem{proposition}{Proposition}[section]

\allowdisplaybreaks

\numberwithin{equation}{section}

\def\R{{\mathbb R}}
\def\Z{{\mathbb Z}}
\def\N{{\mathbb N}}
\def\Exp{{\mathbb E}}
\def\EE{{\cal E}}
\def\Pr{{\mathbb P}}
\def\1{{\mathbf 1}}

\def\sgn{\mathop{\mathrm{sgn}}\nolimits}
\def\eps{\varepsilon}

\def\be{{\mathbf e}}
\def\re{{\mathrm{e}}}
\def\bx{{\mathbf x}}
\def\by{{\mathbf y}}
\def\bu{{\mathbf u}}
\def\bv{{\mathbf v}}
\def\0{{\mathbf 0}}
\def\SS{{\cal S}}

\newcommand{\F}{{\mathcal F}}

\newcommand{\as}{\ \textrm{a.s.}}

\title{Excursions and path functionals for stochastic processes with asymptotically zero drifts\footnote{
We benefitted in the early stages of this project from enjoyable discussions with
Iain  MacPhee, who sadly passed away on 13th January 2012; we dedicate this paper to Iain, in memory of
our valued  colleague and in gratitude for his generosity.}}
\author{Ostap Hryniv\footnote{Department of Mathematical Sciences, University of Durham, South Road, Durham, DH1 3LE, UK.} \and  Mikhail V.\ Menshikov\footnotemark[1]
\and  Andrew R.\ Wade\footnote{Department of Mathematics and Statistics, University of Strathclyde, 26 Richmond Street, Glasgow, G1 1XH, UK.}}

\begin{document}

\maketitle

\begin{abstract}
We study discrete-time stochastic processes $(X_t)$ on $[0,\infty)$ with asymptotically zero mean drifts.
Specifically, we consider the critical (Lamperti-type) situation in which the mean drift at $x$
is about $c/x$. Our   focus is the recurrent case (when $c$ is not too large). We give sharp asymptotics
for various functionals associated with the process and its excursions, including results on maxima and return times.
These results include improvements on existing results in the literature in several respects, and also include
 new results on excursion sums and additive functionals
of the form $\sum_{s \leq t} X_s^\alpha$, $\alpha >0$. We make minimal moments assumptions on the increments of the process.
Recently there has been renewed interest in Lamperti-type process in the context of random polymers and interfaces, particularly nearest-neighbour random walks on the integers;
some of our results are new even in that setting.
We give applications of our results   to processes on the whole of $\R$ and to
a class of multidimensional  `centrally biased'  random walks on $\R^d$; we also apply our results to the simple harmonic urn,
allowing us to sharpen existing results and to verify a
conjecture of Crane {\em et al.}
\end{abstract}



\section{Introduction}

The study of functionals defined on paths of stochastic processes is a topic with classical foundations and extensive applications
in modern probability; such functionals give a quantitative encapsulation of
both probabilistic information about the recurrence behaviour of the process and
 geometrical information about the way in which the process explores the state-space. A substantial body of work is devoted to {\em additive functionals} of the form
$\sum_{s=1}^t \Phi ( X_s )$,
where $X_1, X_2, \ldots$ is a discrete-time stochastic process on $\R^d$ and $\Phi : \R^d \to \R$ is a given measurable function.
The most basic choice, in which $\Phi ( x)$ is taken to be  the indicator function $\1 \{ x \in A\}$
of a Borel set $A \subseteq \R^d$, leads to occupation time for $A$.
In the most well-studied case, $X_t$ is a sum of i.i.d.\ random variables;
the monograph by Borodin and Ibragimov \cite{bi} is devoted to limit theory in the i.i.d.\ setting.
There is also much work devoted to the case in which $X_t$ is an {\em ergodic} Markov
process. Classical work goes back to  Markov
and   Bernstein (cf.\ \cite[p.\ 2299]{koro}). If   $\Phi$ is integrable with respect
to the stationary distribution of the process, then a large collection
of `ergodic theorems' and distributional limit theorems
(after suitable scaling and under various conditions) are known,
and represent an active area of research: see e.g.\
\cite{maxwoo,jko,kv,cscs,koro} and references therein for an indication of the extensive literature.

In the present paper, we study
stochastic processes on $\R^d$ of a more general type (assuming a regenerative property only, rather than the Markov property)
and in the {\em near-critical} situation from the point of view of the asymptotic behaviour of the process. Near-criticality entails that the
one-step mean drift of $X_t$ is {\em asymptotically zero}
in a sense that we describe more precisely later on.
Processes with asymptotically-zero drifts are   of interest in their own right
 for exploring phase transitions in asymptotic behaviour, as first described in the general setting in fundamental work of Lamperti \cite{lamp1,lamp3}.
A  consequence of near-criticality is that many important
distributions associated with such processes display
heavy-tailed behaviour; these include
passage times \cite{aim,lamp3} and, if they exist, stationary distributions \cite{mp}.
In many cases of interest, these and other quantities have natural scaling exponents that
depend on the details of the  process.

Moreover, such processes
are important from the point of view of applications for two main reasons: first, they serve as prototypical near-critical stochastic systems
and hence for the development of new techniques, and second, they can often be extracted from more complex near-critical systems via the method of Lyapunov functions, to powerful analytic effect. One classical but important illustration of the latter point
is provided by the Lyapunov function approach to P\'olya's theorem on the recurrence/transience of symmetric simple
random walk $S_t$ on $\Z^d$:  P\'olya's theorem can be understood in an entirely one-dimensional
setting by taking $X_t = \| S_t\|$, in which case the process $X_t$ has
asymptotically zero drift in the sense that
\[ \Exp [ X_{t+1} - X_t \mid S_t = \bx ] = c_d \| \bx \|^{-1} + O( \| \bx \|^{-2} ) ,\]
for some constant $c_d >0$ that depends only on $d$. Lamperti's recurrence
classification for such processes \cite{lamp1} implies P\'olya's theorem. Importantly, the same
technique works for a very large class of random walks; in particular, the Markov property is not essential.

For these asymptotically zero drift processes, we study a class of additive functionals (or {\em path integrals}) of the form
$\sum_{s=1}^t X_s^\alpha$, $\alpha \geq 0$. Moreover, we study the maximum functional $\max_{1 \leq s \leq t } X_s$
(which corresponds in a certain sense to the $\alpha \to \infty$ limit of the additive functional).
We are interested in the large-$t$ asymptotics of such functionals, in the case where $X_t$ is recurrent.
Our primary interest is not in the case where the process has sufficient `ergodicity' properties that
$t^{-1} \sum_{s=1}^t X_s^\alpha$ converges, but rather in the case where $\sum_{s=1}^t X_s^\alpha$ grows
faster than linearly, including the case where $X_t$ is {\em null-}recurrent. This   is rather different
to the emphasis of the classical work cited above, and accords with our focus on systems that are
{\em near-critical} in some sense.

We make some further remarks on applications and related results. Borovkov {\em et al.} \cite{bbp} consider analogues   for queueing models
of the path integrals that we study.
As far as the authors are aware,
there has been little work specifically concerned with additive functionals of processes with asymptotically zero drift.
For a particular (null-recurrent) example of a nearest-neighbour random walk on the nonnegative integers, Fal' \cite{fal78}
proved distributional limits for functionals such as $\sum_{s=1}^t (1+X_s)^{-\gamma}$, $\gamma > 0$ sufficiently large. Our interest
is in functionals of the opposite nature.

A (normalization of a) particularly important path integral
is the {\em center-of-mass} process associated with $X_t$ defined by $G_t = t^{-1} \sum_{s=1}^t X_s$.
The behaviour of $G_t$ for Markov processes is only partially understood beyond the
case in which sufficient ergodicity ensures that $G_t$ converges to a limit.
For the centre-of-mass associated with simple random walk
on $\Z^d$,
Grill
\cite{grill} proves the interesting result  that (compact-set) recurrence is present if and only
if $d=1$. The desire to understand Grill's result more generally was one of the original
motivations for the work of the present paper;  by analogy with
  Lamperti \cite{lamp1,lamp3}, it is natural
 to begin  in the   setting of processes with asymptotically zero drifts.

Our approach is via a detailed study of the {\em excursions} of the process $X_t$, in which, once again,
the heavy-tailed nature of the characteristics of the processes becomes evident. Thus, if $\eta$ denotes the
duration of an excursion, we are led to the study of {\em excursion functionals} such as $\sum_{s=1}^\eta X_s^\alpha$
(including the special case $\alpha =0$ of $\eta$ itself) and $\max_{1 \leq s \leq \eta} X_s$. As well as being key
ingredients in the proofs of our large-$t$ asymptotics, these quantities are of interest in their own right in various
theoretical and applied contexts. For example, to apply Theorem 2.1 of \cite{cscs} one needs
to understand   tail properties of an analogue of $\sum_{s=1}^\eta X_s^\alpha$;
  sums over excursions for processes with asymptotically zero drift
turn out to be central to the analysis of the `simple harmonic urn' \cite{shu} (see also Section \ref{sec:shu} below).

In the last decade or so, significant interest in processes with asymptotically zero drifts has come from
a community of probabilists and statistical physicists from the point of view of modelling the configurations of polymers
and interfaces. A now standard approach in this field is to take as an underlying model a nearest-neighbour random
walk with an asymptotically zero drift: see for example \cite{alexander,dec,huillet}.
Such nearest-neighbour models are amenable to explicit calculation, often via intricate algebraic
methods such as Karlin--McGregor spectral theory and orthogonal polynomials \cite{karlin}; other recent work on these models, not
directly motivated by polymer models, includes for example \cite{cfr,gallardo,voit90,km}. This continued interest
in asymptotically zero drift processes in the nearest-neighbour case is another motivation for the present paper,
in which we present related results for a much more general class of models. We discuss the relation of our results to some
of this recent work in more detail in Section \ref{sec:poly}.

The outline of the remainder of the paper is as follows. In Section \ref{sec:results} we give a formal statement
of our half-line model and state our results on excursions and functionals in a series of subsections. In Section \ref{sec:appl}
we give applications of our half-line results to processes on the whole line (Section \ref{sec:line})
and to multidimensional processes including centrally biased random walks on $\R^d$ (Section \ref{sec:nonhom})
and the simple harmonic urn (Section \ref{sec:shu}). Also, in Section \ref{sec:poly}, we make some remarks
on how our model and results complement recent results, restricted to nearest-neighbour random walks, in the
context of models of random polymers and interfaces. The proofs of the results in Sections \ref{sec:results}
and \ref{sec:appl}
are given in Sections \ref{sec:proofs} and \ref{sec:proofs2} respectively.

   \section{Main results on path functionals}
   \label{sec:results}

\subsection{Description of the model}

   We formally describe our process  $X:=(X_t)_{t \in \N}$  ($\N:= \{1,2,\ldots\}$) and our
    structural assumptions on its state-space $\SS$. Recall that
   a subset $R$ of $\R^d$ is {\em locally finite} if $R \cap H$ is finite for all bounded $H \subset \R^d$. Our basic assumption is the following.
   \begin{itemize}
   \item[(A0)] (a) Let $\SS$ be a locally finite, unbounded subset of $[0,\infty)$ with $0 \in \SS$.

   (b)
 Suppose that
    $(X_t)_{t \in \N}$ is an $\SS$-valued process adapted to a filtration $(\F_t)_{t \in \N}$, and $\Pr [ X_1 =0 ] =1$.
\end{itemize}

  We also assume the following form of `irreducibility'.

   \begin{itemize}
   \item[(A1)] Suppose that for each $x, y \in \SS$ there exist $m(x,y) \in \N$
   and $\varphi(x,y) >0$ such that
   \begin{equation}
   \label{irred}
    \Pr [ X_{t+m(X_t,y)} = y \mid \F_t ] \geq \varphi (X_t, y), \as , ~{\textrm{for all}} ~ t \in \N .\end{equation}
    \end{itemize}
   If $X$ is a time-homogeneous Markov process, (\ref{irred}) reduces to the
   usual sense of irreducibility that, for any $x, y \in \SS$, there exists
   $m(x,y) \in \N$ such that
   $\Pr [ X_{m(x,y)} =y \mid X_1 =x] >0$.
   The assumption (\ref{irred}) allows us to work with more general processes, such as functions of
   Markov process: see the discussion at the end of this subsection.
   A consequence of (A0) and (A1) is that $\limsup_{t \to \infty} X_t = \infty$, a.s.; see
   Proposition \ref{dichot} below.

  We make some `Lamperti-style' assumptions on the   increments of $X$. Throughout we use the notation
  $\Delta_t := X_{t+1} - X_t$.
  We will typically need to assume that for some $p >2$ (at least), some $\delta > 0$,
  and some constant
  $C \in (0,\infty)$, for all $t\in\N$,
  \begin{equation}
  \label{moms}
  \Exp [ | \Delta_t |^p \mid \F_t ] \leq C (1+X_t)^{p-2-\delta} , \as \end{equation}

   Given that (\ref{moms}) holds for some $p>2$, $\Exp [ \Delta_t^k \mid \F_t ]$
   is a.s.\ finite for $k \in \{1,2\}$. We make some further assumptions on the moments
   of the increments, as follows. For notational convenience, throughout the paper
   we write $\log^q x$ for $(\log x)^q$, $q \in \R$.

   \begin{itemize}
   \item[(A2)]
   Suppose that for some $c \in \R$ and $s^2 \in (0,\infty)$, as $X_t \to \infty$,
  \begin{align}
  \label{mu1}
   \Exp [ \Delta_t \mid \F_t ] & = c X_t^{-1} + o (X_t^{-1} \log^{-1} X_t ), \as, \\
   \label{mu2}
   \Exp [ \Delta_t^2 \mid \F_t ] & = s^2 + o( \log^{-1} X_t ) , \as \end{align}
\end{itemize}
We make an important note on notation: our usage assumes that  implicit constants in
   Landau $O( \; \cdot \; )$,
$o( \; \cdot \; )$ symbols are {\em non-random} and independent of $t$, so that asymptotic expressions such as (\ref{mu1}) and (\ref{mu2})
are understood to hold uniformly in $t$ and probability space elements $\omega$ (on a set of probability 1). So, for example, (\ref{mu2})
means that for any $\eps>0$ we can choose $x < \infty$ so that $| \Exp [ \Delta_t^2 \mid \F_t ] - s^2 |  \leq \eps / \log X_t$, a.s.,
on $\{ X_t > x\}$, for any $t \in \N$.

We study
$X$ via its {\em excursions}
  from $0$.
   Set $\tau_0 := 1$ and for $n \in \N$ define
  \[ \tau_n := \min \{ t > \tau_{n-1} : X_t = 0 \} ,\]
  with the usual convention that $\min \emptyset = \infty$.
  That is,
 $\tau_0, \tau_1, \tau_2, \ldots$ are the successive times of visit to the origin by $X$;
  if $X$ visits $0$ only finitely often then $\tau_n = \infty$ for all $n$ large enough.
  When $\tau_n < \infty$ we denote, for $n \in \N$,
  $\eta_n := \tau_n - \tau_{n-1}$,
  the duration of the $n$th excursion; also set $\eta_0:=1$.
 Provided that $\tau_n < \infty$, we denote the
 $n$th {\em excursion} ($n \in \N$) by $\EE_n :=  (X_t)_{\tau_{n-1} \leq t \leq \tau_{n} -1}$.
  Let $N := \min \{ n \in \N : \tau_n = \infty \}$.

  \begin{itemize}
  \item[(A3)]
(a)
 Suppose that, on $\{ N = \infty \}$,  $(\EE_n)_{n \in \N}$ is an i.i.d.\ sequence.

 (b)
 Suppose that, for all $n \in \N$,
 $\Pr [ \eta_{n+1} < \infty \mid \tau_n < \infty] = \Pr [ \eta_1 < \infty ]$.
  \end{itemize}

 Part (a) of (A3) assumes a full regenerative structure
  in the case of an infinite number of returns to $0$. Part (b) makes a weaker assumption,
  needed to deal with the event $\{ N < \infty \}$.
   A useful reference for regenerative processes is \cite[Chapter VI]{asmussen}.

  Our irreducibility and regenerative assumptions have the following basic consequence.

   \begin{proposition}
   \label{dichot}
   Suppose that (A0) and (A1) hold. Then  $\limsup_{t \to \infty} X_t = \infty$, a.s.
   If, in addition, part (b) of (A3) also holds, then either:
   \begin{itemize}
   \item[(i)] (transience) $\Pr [ \eta_1 < \infty ] <1$ and $\lim_{t \to \infty} X_t = \infty$  a.s.; or
   \item[(ii)] (recurrence) $\Pr [ \eta_1 < \infty ] =1$ and $\liminf_{t \to \infty} X_t = 0$  a.s.
   \end{itemize}
   \end{proposition}

We give the proof of Proposition \ref{dichot} in Section \ref{sec:proofsthms}, along with the proofs of the other results that we present
in the present section.

  Before describing our main
  results, we indicate why we have chosen
  our particular assumptions. It is too restrictive for the applications
  that we have in mind to assume  that $X$ is itself
  a Markov process; our more general framework enables us to work with, for example,   $X_t = \| Y_t \|$
  where $Y_t$ is a Markov process on   $\R^d$. More generally, suppose that $(Y_t)_{t\in\N}$ is an irreducible time-homogeneous
  Markov process
   on an arbitrary countable set $\Sigma$, and let $f: \Sigma \to [0,\infty)$ be measurable such that
   $f^{-1}(x)$ is finite for each $x$.
   Let $\F_t = \sigma ( Y_1, \ldots, Y_t)$ and take $X_t = f(Y_t)$. Then $X_t$
   is $\F_t$-adapted and has the countable state-space $\SS = f(\Sigma)$. Moreover, by irreducibility,
   given  $u$ and $v$ such that $f(u) =x$ and $f(v) =y$, there exist  $\varphi (x, y) >0$
  and $m(x,y) \in \N$ such that $\Pr [ Y_{t+m(x,y)} = v \mid Y_t = u ] \geq \varphi (x,y)$,
  using the fact that, by our assumption on $\Sigma$ and $f$, there are only finitely
  many possible $u,v$ pairs for a given $x,y$. Hence (\ref{irred}) holds.
  If $f^{-1} (0) = 0$ is unique,
  (A3) follows from the strong Markov property.
  This generality
  is very useful for applications: we describe one such example in detail in Section
  \ref{sec:nonhom} below.

The remaining parts of this section are devoted to our results. Our excursion-based approach
is only applicable in the recurrent case, so first we give a recurrence classification. Then we
move on to detailed properties of excursions and the tails of associated random variables,
and finally to $t \to \infty$ asymptotics of
functionals defined on paths of the
process up to some given time $t$.  We exhibit various tail or scaling exponents for the quantities that we study;
we remark that the relationship amongst the various exponents is not always
  correctly predicted by na\"ive heuristic arguments.

   \subsection{Recurrence classification}
   \label{sec:recurrence}

 We say  $X$ is {\em transient} if $\Pr [ \eta_1 < \infty] <1$;
   otherwise it is {\em recurrent}. If recurrent, we say that $X$ is {\em
   positive}-recurrent if $\Exp [ \eta_1] < \infty$ and
   {\em null}-recurrent if $\Exp [ \eta_1]=\infty$.
  Under (A2), the quantity $-2c/s^2$ will play a central role in all that follows, and we
  introduce the notation  \begin{equation}
   \label{rdef} r:= -2c/s^2. \end{equation}
The recurrence classification for $X$ is as follows.

   \begin{theo}
   \label{recthm}
   Suppose that (A0)--(A3) hold, and  (\ref{moms}) holds with $p>2$.
   Then $X$ is \begin{itemize}
   \item[(i)] transient if $r < -1$;
   \item[(ii)] null-recurrent if $-1 \leq r \leq 1$;
   \item[(iii)] positive-recurrent if $r >1$.
   \end{itemize}
   \end{theo}

 Theorem \ref{recthm} is essentially due to
   Lamperti \cite{lamp1,lamp3} (in the case $|r| \neq 1$)
   and Menshikov {\em et al.} \cite{mai} (in the case $|r|=1$) under somewhat different conditions. We do not give a detailed proof of Theorem \ref{recthm} here, but sketch in Section \ref{sec:proofsthms}
   how these existing results may be adapted to our
   setting.

   \subsection{The maximum of an excursion}
   \label{sec:excmax}

   Let
 $M_n := \max_{\tau_{n-1} \leq t < \tau_n} X_t$,  the maximum attained by  $\EE_n$. The next
  result gives tail bounds, in the recurrent case, for the
 $(M_n)_{n \in \N}$, which are i.i.d.\ under our assumptions.

   \begin{theo}
   \label{lem4}
   Suppose that (A0)--(A3) hold.
   Suppose that $r > -1$ and  (\ref{moms}) holds with $p > \max \{ 2, 1+r \}$.
   Then for any $\eps >0$,  for all $x$ sufficiently large,
   \begin{equation}
   \label{mtails}
  x^{-1-r} (\log x)^{-\eps} \leq \Pr [ M_1 \geq x ] \leq x^{-1-r} (\log x)^{1+\eps}.\end{equation}
   In particular,
   $\Exp \left[ M_1^{1+r} \right] = \infty$ but, for any $\eps>0$,
 $\Exp \left[ M_1^{1+r - \eps} \right] < \infty$.     \end{theo}

\begin{rmks}
(a) Symmetric simple random walk on the half-line with reflection at $0$ (and, indeed, any
 of a host of more
   general zero-drift models) has $r=0$, and is thus right on the boundary
   of having a finite expectation for $M_1$.
   (b)
In the tail bound (\ref{mtails}) and similar
results in the sequel, the polynomial term is sharp but we do not
necessarily strive for the best possible logarithmic term. Our results are all sharp enough, however, to classify completely
which moments do or do not exist for the random variable in question.
\end{rmks}

   \subsection{The duration of an excursion}
   \label{sec:exctime}

   The following result is a sharpening in our context of   \cite[Propositions 1 and 2]{aim}, which themselves extended
   work of Lamperti \cite{lamp3}.

   \begin{theo}
   \label{etalem}
 Suppose that (A0)--(A3) hold. Suppose that $r > -1$ and (\ref{moms}) holds
with $p > \max \{2,1+r\}$.
 Then for any $\eps >0$, for all $x$ sufficiently large,
 \begin{equation}
 \label{etatails}
x^{-\frac{1+r}{2}} (\log x)^{-\eps} \leq \Pr [ \eta_1 \geq x ] \leq x^{-\frac{1+r}{2}} (\log x)^{2+r+\eps} .
  \end{equation}
 In particular,
 $\Exp \big[ \eta_1^{\frac{1+r}{2}} \big] = \infty$ but, for any $\eps >0$, $\Exp \big[ \eta_1^{\frac{1+r}{2}-\eps} \big] < \infty$.
   \end{theo}

  The existence of moments for $\eta_1$ part of Theorem \ref{etalem} is based
   on general results of \cite{ai2}. The non-existence of moments result is new
   in the generality given here; under more restrictive assumptions (including uniformly bounded increments for $X_t$) it can be derived from \cite[Corollary 1]{ai1}.    Our proof of the non-existence result is
   based on the intuitively  appealing Lemma \ref{lowbound} below.
   Lamperti \cite{lamp3} was the first to systematically study the problem of the existence or non-existence
   of moments $\Exp [ \eta_1^q]$: his results covered only integer $q$.
   Subsequently Aspandiiarov {\em et al.} extended Lamperti's results to all $q>0$  (see the Appendix of \cite{aim}),
   but neither \cite{lamp3} nor the results of \cite{aim} determine whether the boundary case $\Exp [ \eta^{(1+r)/2} ]$ is finite or infinite;
   as mentioned above, results of \cite{ai1} can be used to settle the boundary case, but under  more restrictive conditions on the increments than we
   use in Theorem \ref{etalem}.
 (The results of \cite{lamp3,aim} related to Theorem \ref{etalem} are stated in the Markovian setting, but their methods, similar to ours, work   more generally.)

\subsection{Number of excursions}
\label{sec:excnum}

Let $N_t$ denote the number of excursions up until time $t$, i.e.,
$N_t := \max \{ n \in \N : \tau_n \leq t \}$.

\begin{theo}
\label{numberthm}
Suppose that (A0)--(A3) hold.
\begin{itemize}
\item[(i)] Suppose that $-1 < r \leq 1$ and   (\ref{moms}) holds
 with $p > 2$. Then for any $\eps >0$,
a.s., for all but finitely many $t$,
\begin{equation}
\label{nbounds}
 t^{\frac{1+r}{2}} (\log t)^{-3 -r -\eps} \leq N_t \leq  t^{\frac{1+r}{2}} (\log t)^{1+\eps} .\end{equation}
\item[(ii)] Suppose that $r>1$  and   (\ref{moms}) holds
 with $p > 1+r$. Then a.s., as $t \to \infty$,
$t^{-1} N_t \to \frac{1}{\Exp [ \eta_1 ]} \in (0,\infty)$.
\end{itemize}
\end{theo}

 \subsection{Occupation times and stationary distribution}
 \label{sec:occupation}

 In this section $\Exp [ \eta_1] < \infty$.
 Define for $t\in \N$ and $x \in \SS$ the occupation times
 $L_t (x) := \sum_{s=1}^t \1 \{ X_s = x \}$.
 Also define the occupation times during the $n$th excursion by
 \begin{equation}
 \label{excocc}
  \ell_n (x) := \sum_{t=\tau_{n-1}}^{\tau_n-1} \1 \{X_t = x \} .\end{equation}
 The next result is essentially a consequence of   `ergodic
 theory' for regenerative processes. The limiting distribution $\pi$ that appears
 in Theorem \ref{lemstat}
 is the usual (unique) stationary distribution if $X$ is an irreducible positive-recurrent
 Markov process.

 \begin{theo}
 \label{lemstat}
 Suppose that (A0)--(A3) hold,   $r >1$,
  and (\ref{moms}) holds
with $p >1+r$.
 Then setting
 \begin{equation}
 \label{pidef} \pi (x) := \frac{\Exp [ \ell_1 (x) ]}{\Exp [ \eta_1 ]} , \end{equation}
 we have that $\pi (x) >0$, $\sum_{x \in\SS} \pi(x) =1$, and,
 for any $x \in \SS$, $t^{-1} L_t (x) \to \pi (x)$ a.s.\ and in $L^q$ for any $q \geq 1$. Finally, if, in addition, the distribution of $\eta_1$ is not
 supported on $k \N$ for any $k >1$, we have that, for any $x \in \SS$,
$\lim_{t \to \infty} \Pr [ X_t =x] = \pi (x)$.
 \end{theo}

 \begin{rmk}
 In the case of a Markov process with uniformly bounded increments,
 under assumptions otherwise similar to ours, results of Menshikov and Popov \cite{mp} show
 that, for $r>1$, $\pi (x) = x^{-r+o(1)}$ as $x \to \infty$. The asymptotics of $\pi(x)$
 are not of direct interest to the topic of the present paper, and so we do not discuss this further here,
 but our methods can be used to
 extend such results to the present more general setting.
 \end{rmk}

\subsection{Running maximum process}
\label{sec:max}

 In this section we consider the process of maxima of $X$, i.e., $\max_{1 \leq s \leq t } X_s$.

\begin{theo}
\label{xbounds}
Suppose that (A0)--(A3) hold.
\begin{itemize}
\item[(i)] Suppose that $-1 < r \leq 1$ and   (\ref{moms}) holds
with $p > 2$.
Then for any $\eps >0$,
a.s., for all but finitely many $t$,
\[ t^{\frac 12} (\log t)^{-\frac{4+r}{1+r} - \eps} \leq \max_{1 \leq s \leq t } X_s
\leq t^{ \frac 12} (\log t)^{\frac{3}{1+r} + \eps }  .\]
\item[(ii)] Suppose that $r>1$  and   (\ref{moms}) holds
with $p > 1+r$. Then for any $\eps >0$,
a.s., for all but finitely many $t$,
\[ t^{ \frac{1}{1+r} } (\log t)^{- \frac{1}{1+r}  - \eps} \leq \max_{1 \leq s \leq t } X_s
 \leq  t^{ \frac{1}{1+r} } (\log t)^{ \frac{2}{1+r}  + \eps} .\]
\end{itemize}
\end{theo}

\begin{rmks}
(a) Related  bounds in a general Lamperti-type setting are given in \cite[Section 4]{mvw};
the excursion-based
approach adopted here has both advantages and disadvantages compared to the method of \cite{mvw}.
The upper bounds in Section 4 of \cite{mvw}
essentially apply in the present setting (concretely, use
\cite[Theorem 3.2]{mvw} with Lemma \ref{semilem} here),
and lead to slightly sharper upper bounds than those in our Theorem \ref{xbounds}.
(See also Section 6 of \cite{bary} for some variations on these upper bounds.)
However, the lower bounds in \cite{mvw}
  cannot  readily be applied
here, even assuming a uniform bound on the increments of $X_t$.
Thus our lower bounds in Theorem \ref{xbounds} represent
progress over previous results.

(b) Our excursion-based approach  sheds no light on the transient case $r < -1$.
  For $r < -1$, under several additional assumptions, \cite[Theorem 4.2]{mvw} shows that there exists $D \in (0,\infty)$ such that
a.s., for all but finitely many $t \in \N$,
$X_t \geq t^{1/2} (\log t)^{-D}$.
  This result can be viewed
as a  generalization of the classical Dvoretzky--Erd\H os theorem
on rate of escape of transient simple symmetric random walk in $\Z^d$ $(d \geq 3)$ \cite{de}.

(c) In various special cases of certain nearest-neighbour random walks on $\Z^+$, using methods restricted to the nearest-neighbour case,
sharper versions of one or other of the bounds in Theorem \ref{xbounds}(i) are given in
\cite{rosenkrantz,fal81,szekely,gallardo,voit90,huillet}; of these, only \cite{huillet} also has a version of Theorem \ref{xbounds}(ii).
\end{rmks}

 \subsection{Single-excursion sums}
 \label{sec:excsum}

 For $\alpha \geq 0$ and $n \in \N$ set
 \begin{equation}
 \label{xidef}
  \xi^{(\alpha)}_n := \sum_{t= \tau_{n-1}}^{\tau_n -1} X_t^\alpha
 = \sum_{x \in \SS} x^\alpha \ell_n (x) ,\end{equation}
 with the occupation time notation of (\ref{excocc});
note that $\xi^{(0)}_n = \eta_1$.
Our next result gives tail bounds for $\xi^{(\alpha)}_1$.
Theorem \ref{xilem} has
applications in its own right:
 for example in \cite[Theorem 2.1, p.\ 908]{cscs} one
  is required to verify
 a condition  similar to $\Exp [ (\xi_1^{(\alpha)})^{2+\delta}] < \infty$.

 \begin{theo}
 \label{xilem}
 Suppose that (A0)--(A3) hold. Suppose that $r > -1$ and  (\ref{moms}) holds
with  $p > \max \{2,1+r\}$.
 Let $\alpha \geq 0$.
 Then for any $\eps >0$, for all $x$ sufficiently large,
 \begin{equation}
 \label{xitails}
  x^{-\frac{1+r}{\alpha+2}} (\log x)^{-\eps} \leq \Pr [ \xi^{(\alpha)}_1 \geq x] \leq  x^{-\frac{1+r}{\alpha+2}} (\log x)^{\frac{2+2r}{\alpha+2}+1 +\eps}.\end{equation}
 In particular, $\Exp \big[ (\xi^{(\alpha)}_1)^{\frac{1+r}{\alpha+2}} \big] = \infty$ but, for any $\eps>0$,
 $\Exp \big[ (\xi^{(\alpha)}_1)^{\frac{1+r}{\alpha+2} -\eps} ] < \infty$.
 \end{theo}

\begin{rmks} (a) The $\alpha=0$ case of Theorem \ref{xilem} reduces to Theorem \ref{etalem}.
 Theorem \ref{xilem} can also be seen as a generalization of Theorem \ref{lem4},
 since here $\lim_{\alpha\to\infty} ( \xi_1^{(\alpha)} )^{1/\alpha} = M_1$, a.s.,
 so for any $x$, $\Pr [ \xi_1^{(\alpha)} \geq x^\alpha ] \to \Pr [ M_1 \geq x]$ as $\alpha \to \infty$.

 (b) For simplicity we have stated our results for functionals based on   $x \mapsto x^\alpha$,
 but our methods apply  to any nonnegative nondecreasing function (cf Lemma \ref{lowbound} below).
 \end{rmks}

 \subsection{Path integrals}
 \label{sec:integrals}

   Fix $\alpha \geq 0$ and define $S^{(\alpha)}_t := \sum_{s =1}^t X_s ^\alpha$.
  We have the following  asymptotic results on $S^{(\alpha)}_t$.

 \begin{theo}
 \label{sthm}
 Suppose  (A0)--(A3) hold, $r > -1$, and  (\ref{moms}) holds
with $p > \max \{2,1+r\}$.
 \begin{itemize}
 \item[(i)]
Suppose that $-1 < r \leq 1$. Then
  for any $\eps >0$, a.s., for all but finitely many $t$,
  \[ t^{\frac{\alpha+2}{2}} (\log t)^{-\frac{(\alpha+2)(4+r)}{1+r}  -\eps} \leq
 S^{(\alpha)}_{t} \leq t^{\frac{\alpha+2}{2}} (\log t)^{\frac{3\alpha+6}{1+r} + 2 + \eps} .\]
 \item[(ii)]
  Suppose that $1 < r \leq 1+\alpha$. Then
 for any $\eps >0$, a.s., for all but finitely many $t$,
 \[ t^{\frac{\alpha+2}{1+r}} (\log t)^{-\frac{\alpha+2}{1+r} -\eps}
 \leq S^{(\alpha)}_{t} \leq t^{\frac{\alpha+2}{1+r}} (\log t)^{\frac{2\alpha+4}{1+r} + 2 +\eps} .\]
 \item[(iii)] Suppose that $r>1+\alpha$.  Then, with $\pi$
  as defined at (\ref{pidef}),
 as $t \to \infty$, a.s.,
 \begin{equation}
 \label{nudef}  t^{-1} S^{(\alpha)}_{t} \to  \frac{\Exp [ \xi^{(\alpha)}_1]}{\Exp [ \eta_1]} =   \sum_{x \in \SS} x^\alpha \pi(x) =: \nu_\alpha \in (0,\infty) . \end{equation}
 \end{itemize}
 \end{theo}

  Theorem \ref{sthm}(iii)
 is  essentially a consequence of  `ergodic theory'
 for regenerative processes (see e.g.\ \cite[Theorem VI.3.1, p.\ 178]{asmussen})
 but our proof of Theorem \ref{sthm}(i)--(ii) yields part (iii) at little
 additional effort, so we give the self-contained proof in Section \ref{sec:proofsthms}.

  A case of special interest is when $\alpha=1$, in which case it
   is natural to study the normalized sum $t^{-1} S^{(1)}_t$ which is just
   the {\em centre of mass} of $( X_1,\ldots,X_t )$. Denote
   \begin{equation}
   \label{Gdef}
    G_t := t^{-1} S^{(1)}_t = t^{-1} \sum_{s=1}^t X_s .\end{equation}
Theorem  \ref{xbounds} yields the following immediate corollary for $G_t$. For simplicity of presentation,
we suppress the logarithmic factors in Theorem \ref{xbounds} by stating Corollary \ref{gthm} parts (i) and (ii) on the logarithmic scale.

 \begin{cor}
 \label{gthm}
  Suppose  (A0)--(A3) hold, $r > -1$, and  (\ref{moms}) holds
with $p > \max \{2,1+r\}$.
  \begin{itemize}
  \item[(i)] Suppose   that $-1 < r \leq 1$. Then  $\lim_{t\to \infty} \frac{\log G_t}{\log t} = \frac{1}{2}$, a.s.
 \item[(ii)] Suppose that $1 <r \leq 2$. Then $\lim_{t\to\infty} \frac{\log G_t}{\log t} = \frac{2-r}{1+r} \in [0, 1/2)$, a.s.
 \item[(iii)] Suppose that  $r>2$.   Then for $\nu_1 \in (0,\infty)$
 given by (\ref{nudef}), $\lim_{t\to\infty} G_{t} =  \nu_1$, a.s.
\end{itemize}
 \end{cor}

\begin{rmk}
Comparing the scaling exponents in Corollary \ref{gthm} to those in Theorem \ref{xbounds}, we see that they coincide (taking
value $\frac12$) in the null-recurrent case, but   differ  in the positive-recurrent case
($\frac{2-r}{1+r} < \frac{1}{1+r}$ for $r>1$). The intuition here is that
in the positive-recurrent case, the process rarely visits the scale of the maximum, so $G_t \ll \max_{1 \leq s \leq t} X_s$.
\end{rmk}

   \section{Applications}
\label{sec:appl}

 \subsection{Processes on the whole real line}
 \label{sec:line}

In this section
we give applications of our results from Section \ref{sec:results}
on half-line processes
to models defined on the whole line, for which new phenomena emerge.
We restrict to the Markovian case for simplicity of statement.
The $\R$-valued processes that we study are, loosely speaking,
 two half-line processes sewn together at $0$.

 \begin{itemize}
 \item[(B0)] Let $(X_t)_{t \in \N}$ be an irreducible, time-homogeneous Markov chain on  $\SS$,
 a locally finite subset of $\R$ with $0 \in \SS$, $\inf \SS = - \infty$, and $\sup \SS = + \infty$.
Take $X_1 =0$.
    \item[(B1)] Suppose that $ \Pr [ X_{t+1} = y \mid X_t = x ] = 0$ if $x$ and $y$ are separated by $0$. Suppose also that $\Pr [ X_{t+1} < 0 \mid X_t = 0] \in (0,1)$
    and $\Pr [ X_{t+1} > 0 \mid X_t = 0] \in (0,1)$.
    \end{itemize}

    Under (B1),   $X_t$ cannot jump over the origin, and from the origin jumps left or right each with positive probability.
    As above,   write $\Delta_t := X_{t+1} - X_t$ for the increments of $X_t$.

    \begin{itemize}
        \item[(B2)] Suppose that for some $p>2$ and $\delta>0$, $\Exp [ | \Delta_t |^p \mid X_t = x ] = O (|x|^{p-2-\delta})$
        as $|x| \to \infty$.
         Suppose also that  for some $c_+, c_- \in \R$ and $s^2_+, s^2_- \in (0,\infty)$,
  \begin{align}
  \label{mu1two}
   \Exp [ \Delta_t \mid X_t = x ] & = |x|^{-1} \left(  c_+ \1 \{ x>0 \} - c_- \1 \{ x < 0\} \right)  + o ( |x|^{-1} \log^{-1} |x| ),  \\
    \label{mu2two}
     \Exp [ \Delta_t^2 \mid X_t = x ] & =   \left( s^2_+ \1 \{ x>0 \} + s^2_- \1 \{ x < 0\}  \right)  + o (   \log^{-1} |x| ).
      \end{align}
  \end{itemize}

  Analogously to the definition of $r$ at (\ref{rdef}), set $r_\pm := -2c_\pm/s_\pm^2$.
  In this section we restrict
   to the setting in which $r_-, r_+ \in (-1,1]$, i.e., corresponding to
  null-recurrence of each of the half-line processes. Cases where one or more of $r_-, r_+$ is greater than $1$ can be dealt with using similar
  methods. We assume that $-1 < r_+ < r_- \leq 1$, so that the positive half-line is `less recurrent'.
  The following result demonstrates the interesting
  phenomenon of a separation of scales for the two sides of the process.

  \begin{theo}
  \label{line1}
  Suppose that (B0)--(B2) hold, and that $-1 < r_+ < r_- \leq 1$. Then $X_t$ is null-recurrent, and, a.s.,
  \begin{align*} \lim_{t \to \infty} \frac{ \log \max_{1 \leq s \leq t} X_s }{\log t} & = \frac{1}{2}, \textrm{ and }
  \\
   \lim_{t \to \infty} \frac{ \log | \min_{1 \leq s \leq t} X_s | }{\log t} & = \frac{1}{2} \left( \frac{1+r_+}{ 1+r_-} \right) \in (0,1/2) .
   \end{align*}
\end{theo}

As a concrete example, consider a nearest-neighbour random walk on $\Z$ which jumps as a symmetric simple random walk when on the nonnegative integers,
but from $x < 0$ jumps to $x \pm 1$ with probabilities $\frac{1}{2} \pm \frac{1}{4x}$. Then $r_+ =0$ and $r_- =1$; viewed separately the two
half-line process are null-recurrent and have the same (diffusive) scale, but the `combined' process has scales $t^{1/2}$ on
$[0,\infty)$ and
$t^{1/4}$ on $(-\infty,0]$.

The intuition behind Theorem \ref{line1} is that  the walk makes a comparable number of positive and negative excursions,
but the positive ones have heavier-tailed durations, so occupy a dominant proportion of time.
The same intuition is behind the next result, which shows that the positive sojourns dominate the path-integral asymptotics.
Again we use the notation (\ref{Gdef}), now for $X_s$ taking values in $\R$.

\begin{theo}
\label{line2}
 Suppose that (B0)--(B2) hold, and  $-1 < r_+ < r_- \leq 1$.  Then, a.s.,  $G_t \to + \infty$  and
 \[ \lim_{t\to \infty} \frac{ \log G_t}{\log t} = \frac{1}{2} . \]
\end{theo}

\begin{rmks}
(a) We leave largely open the case $r_+ = r_-$, but see the $d=1$ case of the model in Section \ref{sec:nonhom}.
(b) Similar results to those in this section can be obtained for processes on a state space that consists
of multiple copies of $[0,\infty)$, joined at a common origin, and embedded in $\R^d$.
\end{rmks}

   \subsection{Centrally biased random walks on $\R^d$}
   \label{sec:nonhom}

In this section we work in $\R^d$, $d \in \N$.
For $\bx \in \R^d$,  write
$\bx = (x_1, \ldots,x_d)$ in Cartesian coordinates.
 Let $\| \cdot \|$   denote the Euclidean
norm on $\R^d$. For a non-zero
vector $\bx \in \R^d$ we write $\hat \bx := \bx / \| \bx \|$ for the
corresponding unit vector. Write $\0 := (0,\ldots,0)$
for the origin.

\begin{itemize}
\item[(C0)]
Let $\Xi = (\xi_t)_{t \in \N}$ be an irreducible,  time-homogeneous
Markov process whose state space $\Sigma$ is  an unbounded, locally finite subset of $\R^d$ containing $\0$.
Let $\xi_1 = \0$.
\end{itemize}

We use the notation $\theta_t := \xi_{t+1} - \xi_t$ for the increments of the walk.
The assumption (C0) implies that
the distribution of
$\theta_t$ depends only on the position $\xi_t \in \Sigma$ and not on $t$.
We assume that for some $p >2$, $\delta >0$, and $C < \infty$,
\begin{equation}
\label{moms2}
  \Exp [ \| \theta_t \|^{p} \mid \xi_t = \bx ] \leq C  ( 1 +  \| \bx \| )^{p-2-\delta}.\end{equation}

Denote the one-step
{\em mean drift} vector
$\mu (\bx) := \Exp [ \theta_t \mid \xi_t = \bx ]$
for $\bx \in \Sigma$, and denote the
covariance matrix at $\bx \in \Sigma$ by
$M (\bx) := (M_{ij} (\bx ))_{i,j} := \Exp [ {\theta_t}^{\!\!\top} \theta_t \mid \xi_t = \bx]$,
for $\bx \in \Sigma$, where $\theta_t$ is viewed as  a row-vector.
In   vector equations such as the equation for $\mu (\bx)$
 in the following assumption, an expression of the form $o( h( \| \bx \|))$ is to be interpreted
a   vector   whose components are each $o( h( \| \bx \|))$ as $\| \bx \| \to \infty$, uniformly in $\bx$
given $\| \bx \|$.

\begin{itemize}
\item[(C1)] Suppose that there exist $\rho \in \R$ and $\sigma^2 \in (0,\infty)$ for which, as $\| \bx \| \to \infty$,
\begin{align*}
\mu (\bx ) & = \rho \hat \bx \| \bx \|^{-1} + o ( \| \bx \|^{-1} \log^{-1} \| \bx \| ) ,\\
M_{ij} (\bx ) & = \sigma^2 \1 \{ i = j \} + o ( \log^{-1} \| \bx \| ) .\end{align*}
\end{itemize}

The assumption on $M$ in (C1) implies that $\xi_t$ has an asymptotically diagonal covariance structure.
Processes satisfying (C0) and (C1) were studied by Lamperti \cite{lamp1,lamp3} under the name {\em centrally biased
random walks}, due to the nature of the drift field;
the name had been used earlier by Gillis \cite{gillis} for a different model. Our main result on such models is the following, which will enable us to apply the
results of Section \ref{sec:results} to generalize and sharpen Lamperti's results, among
 other things.

\begin{theo}
\label{cbrw}
Suppose that (C0) and (C1) hold, and (\ref{moms2}) holds for some $p>2$.
Let $X_t = \| \xi_t \|$. Then $X_t$ satisfies the conditions (A0)--(A3), with
\[ c = \rho + (d -1)(\sigma^2/2), ~~~ s^2 = \sigma^2 ;\]
hence $r = 1 - d - (2 \rho/\sigma^2)$.
Moreover,   (\ref{moms}) holds for the given $p>2$.
\end{theo}

From Theorem \ref{cbrw}, we  immediately deduce a series of results for $\Xi$ from the theorems
in Section \ref{sec:results}. We state two such corollaries.
 Note that if we set
$\eta := \min \{ t \in \N : \xi_t = \0 \}$,
we have from (C0) that $\eta = \eta_1$ for $X_t = \| \xi_t \|$ in our previous notation,
since $X_t = 0$ if and only if $\xi_t = \0$. Theorem \ref{cbrw} with Theorems \ref{recthm} and \ref{etalem}
gives the following result.

\begin{cor}
\label{cor1}
Suppose that (C0) and (C1) hold, and  (\ref{moms2}) holds
with $p>2$. Then $\Xi$ is
\begin{itemize}
\item[(i)] transient if $2 \rho / \sigma^2 > 2-d$;
\item[(ii)] null-recurrent if $-d \leq 2 \rho / \sigma^2 \leq 2 -d$;
\item[(iii)] positive-recurrent if $2 \rho/\sigma^2 < - d$.
\end{itemize}
Moreover, in the recurrent cases, $\Exp [ \eta^q ] < \infty$ if and only if $q < q_0:= 1 - (d/2) - (\rho/\sigma^2)$.
\end{cor}

Corollary \ref{cor1} extends   results of Lamperti \cite{lamp1,lamp3}, who
   assumed uniformly bounded increments for $\xi_t$  and a
    stronger version of (B1) with the error term $\log^{-1} \| \bx \|$  replaced by $\| \bx \|^{-\delta}$ for $\delta>0$:
 see Theorem 4.1 of \cite[p.\ 324]{lamp1} and Theorem 5.1 of \cite[p.\ 142]{lamp3}. Also, Lamperti's result only covers {\em integer} $q$, and
is not sharp enough to determine whether $\Exp [\eta^{q_0}]$ is finite or infinite, and so cannot decide
on null- or positive-recurrence at the boundary case $2 \rho /\sigma^2 = -d$.

The next result follows from Theorem \ref{cbrw} with Theorem \ref{xbounds}, and gives almost-sure
scaling behaviour for the maximum of $\| \xi_t \|$ in the recurrent cases.

\begin{cor}
\label{cor2}
Suppose that (C0) and (C1) hold. \begin{itemize}
\item[(i)] Suppose that $-d \leq 2 \rho/\sigma^2 < 2-d$ and (\ref{moms2}) holds with $p>2$.
Then
\[  \lim_{t \to \infty} \frac{ \log \max_{1 \leq s \leq t} \| \xi_s \|}{\log t} = \frac{1}{2}, \as \]
\item[(ii)] Suppose that $2 \rho/\sigma^2 < -d$ and (\ref{moms2}) holds with $p>2 - d - (2\rho/\sigma^2)$.
Then
\[   \lim_{t \to \infty} \frac{ \log \max_{1 \leq s \leq t} \| \xi_s \|}{\log t} = \frac{1}{2 - d - (2\rho/\sigma^2)}, \as \]
\end{itemize}
\end{cor}

Upper bounds similar to those in Corollary \ref{cor2} can be derived from  \cite[Section 3]{mvw}: see Theorem 2.4 of \cite{bary}
for a similar application of such results, albeit under more restrictive assumptions.
As far as the authors are aware,
the lower bounds in Corollary \ref{cor2} are  new.

   \subsection{The simple harmonic urn}
   \label{sec:shu}

In this section we study a particular Markov chain $(A_t,B_t)$ on $\Z^2 \setminus \{ (0,0) \}$, with discrete time $t \in \N$. The model
was introduced in \cite{shu}, motivated by an urn model.
The model takes as input the distribution of a $\Z$-valued random variable $\kappa$. We assume that, for some $\lambda >0$, $\Exp [ \re^{\lambda |\kappa| } ] < \infty$.
Let $\kappa_0, \kappa_1, \ldots$ be a sequence of independent
copies of $\kappa$.
The transition law of the chain is as follows. If $A_t B_t \neq 0$, i.e., the chain is not on one of the coordinate axes,
it takes jumps of unit size according to the following:
\begin{align*}
 \Pr \left[ (A_{t+1}, B_{t+1} )   = ( a, b + \sgn (a) ) \mid (A_t,B_t) = (a,b) \right] & = \frac{|a|}{|a| + |b|}, ~( ab \neq 0); \\
 \Pr \left[ (A_{t+1}, B_{t+1} )   = ( a - \sgn (b) , b) \mid (A_t,B_t) = (a,b) \right] & = \frac{|b|}{|a| + |b|} , ~(ab \neq 0), \end{align*}
 where $\sgn(x) := x/|x|$ for $x \in \R \setminus \{0\}$. From one of the axes, the process jumps as follows:
 \begin{align*}
 (A_{t+1} , B_{t+1} ) & = ( \sgn( A_t) \max \{ 1 , | A_t | - \kappa_t \}   , \sgn (A_t) ) , ~ (A_t \neq 0, B_t = 0 ) ; \\
 (A_{t+1} , B_{t+1} ) & = ( - \sgn(B_t)  , \sgn ( B_t) \max \{ 1, | B_t | - \kappa_t \} ) , ~ (A_t = 0, B_t \neq 0 ) . \end{align*}
In words, the process has an approximately anti-clockwise trajectory,  traversing each quadrant in sequence.
When not on an axis, the process traverses the current quadrant using unit steps in two possible
directions, while from an axis, the process moves one step away from the axis (in the anti-clockwise direction)
and makes a special jump of size distributed as $\kappa$ towards the next destination axis, truncating so as to ensure
it does not actually reach the next axis in this jump.

So defined, $(A_t, B_t)$ is an irreducible Markov chain  on $\Z^2 \setminus \{ (0,0) \}$.

The basic case has $\kappa = 0$ a.s., in which case the process is the {\em simple harmonic urn}; another
particular case has $\kappa = 1$ a.s., which is known as the {\em leaky urn} \cite{shu}. The general $\kappa$ model is known as
the {\em noisy urn} \cite{shu}.
In fact, the leaky urn in \cite{shu} was defined slightly differently, with an absorbing state when $|A_t | + |B_t| =1$,
but the two definitions coincide up until the time of absorption.

Let $\nu_0 := 0$ and, for $n \in \N$, $\nu_{n} := \min \{ t > \nu_{n-1} : A_t B_t = 0 \}$, so that $\nu_1, \nu_2, \ldots$
are the successive times of visits to the axis by the process $(A_t, B_t)$. Define the embedded process
 $Z_t := | A_{\nu_{t}} | + | B_{\nu_{t}} |$ for $t \in \N$;
by construction, exactly one of $| A_{\nu_{t}} |$ and $| B_{\nu_{t}} |$ is 0.
Then $Z_t$ is an irreducible Markov chain on  $\N$,
representing the distance of the original Markov chain from the origin at those times when it visits an axis.
For definiteness, we take $(A_1,B_1 ) = (1, 0)$, so $Z_1 = 1$.

The following result shows the connection between this model
and our present setting.

\begin{proposition}
\label{shuprop}
Let $X_t = \sqrt{Z_t-1}$. Then  (A0)--(A3) hold with $\SS = \{ \sqrt{x -1} : x \in \N \}$,
$c = \frac{1-2\Exp [\kappa]}{4}$, and $s^2 = \frac{1}{6}$;
hence  $r = 6 \Exp[\kappa] -3$.
In addition, (\ref{moms}) holds for any $p>0$.
\end{proposition}

Proposition \ref{shuprop} is closely related to Lemma 7.7 in \cite{shu}, but differs slightly as our embedded process $Z_t$ is
not quite the same as the one used in \cite{shu}, so we sketch the proof in Section \ref{shuproofs} below.
For the original process, we are interested in $\tau := \min \{ t \in \N : | A_t | + | B_t| =1 \}$.
For our embedded process, define $\tau_q := \min \{ t \in \N : Z_t = 1 \}$ (where the `$q$' indicates `quadrant time').
 The key relationship between the two processes is that $\tau = \nu_{\tau_q}$, since $|A_t| + |B_t| =1$ if and only if $t = \nu_k$ for some $k$ and $Z_k = 1$.
In \cite{shu}, a slightly different version of the embedded process $Z_t$ (namely, $\tilde Z_k$ defined on p.\ 2125 of \cite{shu}) was used; for that version
the analogous claim to `$\tau = \nu_{\tau_q}$' made just below equation (6) in \cite{shu} is not  correct as stated, although this has no impact on the results in \cite{shu}. It is not hard to fix this small gap in the argument in \cite{shu}, and the variation given in the present paper is just one way of doing so. More importantly, the results of the present
paper enable us to sharpen the results in \cite{shu} and to settle a conjecture made in that paper.

Proposition \ref{shuprop} enables us to determine the tails of $\tau_q$; some additional work is needed
to account for the change of time between $(A_t, B_t)$ and $X_t$ and hence study the tails of $\tau = \nu_{\tau_q}$.
Due to the special structure of the paths of the simple harmonic
urn process, it turns out that exactly relevant to this point is an excursion sum of the type $\xi_1^{(2)}$ defined by (\ref{xidef}).
We prove the following result in Section \ref{shuproofs}. The condition $\Exp[ \kappa ] > \frac{1}{3}$ corresponds
to $r >-1$, in which case the process is {\em recurrent}.

\begin{theo}
\label{shuthm}
Suppose that $\Exp [ \kappa] > \frac{1}{3}$.
Let $p \geq 0$.
Then $\Exp [ \tau^p ] < \infty$ if and only if $p < \frac{3\Exp[\kappa]-1}{2}$. In particular,
the Markov chain $(A_t, B_t)$ is null-recurrent when $\Exp [ \kappa ] =1$.
\end{theo}

This result shows that $\Exp[\tau ^p]=\infty$ for $p = \frac{3\Exp[\kappa]-1}{2}$,
the boundary case not covered by Theorem 2.6 of \cite{shu}; the fact that the process
is {\em null-}recurrent when $\Exp[\kappa]=1$ confirms the conjecture after Corollary 2.7 in \cite{shu}.
Theorem \ref{shuthm} has the following immediate corollary, which fills the gap in Theorem 2.3 of \cite{shu}.

\begin{cor}
For the leaky   urn, the  time to absorption is non-integrable.
\end{cor}

\subsection{Random walk models of polymers and interfaces}
\label{sec:poly}

The last decade or so has seen   renewed interest in one-dimensional random walks with
asymptotically zero drifts from a statistical physics perspective, concerning models of random polymers and
interfaces,
their structure, and their interactions with a medium or boundary. In the context of random
polymers, the path of the process models the physical polymer chain; the asymptotically zero
drift indicates the presence of {\em long-range interaction} with a boundary, which can be either
attractive or repulsive. For a random interface, the walk models the behaviour
of a liquid interface on a solid substrate (including {\em wetting} and {\em pinning} phenomena);
in this context the drift may represent {\em affinity} for the boundary.
We refer to \cite{holl,giac,velenik} for recent surveys.

Much of the existing work  is restricted to nearest-neighbour random walks on $\Z^+$,
where explicit calculations are facilitated by reversibility and associated algebraic structure (such
as Karlin--McGregor theory \cite{karlin}); see e.g.\ \cite{alexander,dec,huillet} for
models inspired directly by random polymers, and e.g.\ \cite{cfr,gallardo,voit90} for
related work. In this section we make some brief remarks emphasizing how the present paper adds
to this literature, and in particular how our results can be used to study quantities of interest in this context
for a much more general class of processes; our results not only do not require the nearest-neighbour assumption, but do not need bounded jumps
or even the Markov property {\em per se}.

A typical family of nearest-neighbour random walks $X_t$ on $\Z^+$ that has been extensively studied has
$\Pr (X_{t+1} = X_t \pm 1 \mid X_t = x) = \frac{1}{2} \mp \frac{\delta}{4x + 2 \delta}$ for $x>0$ and a parameter $\delta$;
here
  (A2) holds with $c = -\delta/2$ and $s^2 =1$, so $r = \delta$. This and closely related models were considered by Karlin and McGregor \cite{karlin},
and  by many subsequent authors,
including for instance \cite{rosenkrantz,fal71,fal78,fal81,szekely,gallardo,voit90} and, most recently  \cite{dec} and   \cite{huillet}.
In these very special cases, Fal' \cite{fal71} gives asymptotics for excursion times and the number of excursions (cf our Theorem \ref{numberthm}),
while several authors \cite{rosenkrantz,fal81,szekely,gallardo,voit90} give iterated-logarithm type upper bounds in the diffusive case (cf our Theorem \ref{xbounds}(i)).
Huillet \cite{huillet} gives sharper versions of our Theorems \ref{lem4}, \ref{etalem}, and \ref{xbounds}
in this special case: see Propositions 2, 9, 10, and 11 of \cite{huillet}.
The main result of \cite{dec} (see also Proposition 15 of \cite{huillet}) is that, for
$\delta \in (1,2)$, $\Exp [ X_t] \sim K_\delta t^{1 - \frac{\delta}{2}}$, being one possible measure of the spatial extent of the polymer. Perhaps more natural (certainly more
readily interpreted in terms of path properties) are the
quantities $\max_{1 \leq s \leq t} X_s$ and $t^{-1} \sum_{s=1}^t X_s$ that we study in the present paper;
 their scaling exponents for the case $\delta \in (1,2)$ are $\frac{1}{1+\delta}$ (our Theorem \ref{xbounds},
 or Proposition 10 of \cite{huillet}) and $\frac{2-\delta}{1+\delta}$ (our Corollary \ref{gthm}) respectively.
 Note that for $\delta \in (1,2)$,
$\frac{1}{1+\delta} > 1 - \frac{\delta}{2} > \frac{2-\delta}{1+\delta}$.
 Alexander \cite{alexander} calls such nearest-neighbour random walks with drift $O (1/x)$  at $x$ `Bessel-like',
 and gives sharp results on the asymptotics of return times, among  other things.
  There seems to have as yet been no success in applying the methods of \cite{alexander,huillet,dec} beyond the
 nearest-neighbour setting.

 \section{Proofs of main results}
 \label{sec:proofs}

   \subsection{Lyapunov functions}
   \label{sec:lyap}

   For $\gamma, \nu \in \R$ we define the function $f_{\gamma,\nu}:[0,\infty) \to [0,\infty)$
by   \[ f_{\gamma,\nu} (x) := (\re+x)^\gamma \log^\nu (\re+x) . \]
    Our basic analytical method
   will be built on the fact that $f_{\gamma,\nu} (X_t)$ is a submartingale
   or supermartingale, for $X_t$ outside some bounded set,
   for appropriate $\gamma$ and $\nu$. The following result
   is fundamental. The idea here is not new, although the particular form of the result
   is a little different from previous versions in the literature.
Recall that in expressions such as (\ref{semieq4}), the $o(1)$ term is uniform in $t$ and $\omega$.

    \begin{lm}
   \label{semilem}
   Suppose that (A0) and (A2) hold,  $r > -1$,
    and   (\ref{moms}) holds for some
   $p>\max\{2,1+r\}$. Then for any $\nu \in \R$, as $X_t \to \infty$, a.s.,
   \begin{align}
\label{semieq4}
 \Exp [    f_{1+r,\nu} (X_{t+1}  ) - f_{1+r, \nu} (X_t)   \mid \F_t ]
 =  \left( \nu (1+r) (s^2/2)   + o(1) \right) X_t^{r-1} \log ^{\nu -1} X_t .\end{align}
In particular,  for any $\nu >0$,  there exists $A < \infty$
   such that,   on $\{ X_t \geq A \}$, a.s.,
   \begin{align*}
   \Exp [ f_{1+r,\nu} (X_{t+1} ) - f_{1+r, \nu} (X_t ) \mid \F_t ] & \geq 0 ;\\
    \Exp [ f_{1+r,-\nu} (X_{t+1} ) - f_{1+r, -\nu} (X_t ) \mid \F_t ] & \leq 0 .
    \end{align*}
\end{lm}

Before  proving Lemma \ref{semilem}, we state a technical result.
 For $\eps \in (0,1)$, let $E_\eps(t)$ denote the  event
    $\{ | \Delta_t | \leq (1+X_t)^{1-\eps} \}$. Denote
    the complementary event by $E^{\rm c}_\eps(t)$.

   \begin{lm}
   \label{deltail}
   Suppose that (\ref{moms}) holds with $p>2$ and $\delta >0$. Then for some $C \in (0,\infty)$ and
   any $\eps \in (0, \frac{\delta}{1+p})$, for any
   $q \in [0,p]$,
   $\Exp [ | \Delta_t |^q \1_{E_\eps^{\rm c} (t)} \mid \F_t ] \leq C (1+X_t)^{q - 2- \eps}$, a.s.
   \end{lm}
   \begin{proof}
   For $q \in [0,p]$, $|\Delta_t |^q \1_{E_\eps^{\rm c} (t)} = | \Delta_t |^p | \Delta_t |^{q-p} \1_{E_\eps^{\rm c} (t)} \leq
   | \Delta_t |^p (1+ X_t)^{(1-\eps)(q-p)}$, by definition of $E_\eps (t)$.
   Taking expectations and using (\ref{moms}) we obtain
   \[ \Exp [ | \Delta_t |^q \1_{E_\eps^{\rm c} (t)} \mid \F_t ] \leq C (1+X_t)^{p-2-\delta+(1-\eps)(q-p)} ,\]
   and the result follows.
    \end{proof}

   \begin{proof}[Proof of Lemma \ref{semilem}.]
   Let $\gamma \geq 0$ and $\nu \in \R$.    Take $\eps \in (0,1)$.
   We estimate the expected increment of $f_{\gamma, \nu} (X_t)$
   using a Taylor expansion on $E_\eps(t)$, while we use Lemma \ref{deltail} to control
   the expectation on $E^{\rm c}_\eps(t)$. Differentiation of $f_{\gamma,\nu}$ with respect to $x$
   gives
   \begin{align*}
   f'_{\gamma,\nu} (x ) & = \gamma f_{\gamma-1,\nu} (x) + \nu f_{\gamma-1,\nu-1} (x) ; \\
f''_{\gamma,\nu} (x) & = \gamma (\gamma-1) f_{\gamma -2,\nu} (x) + \nu (2\gamma-1) f_{\gamma-2,\nu-1} (x) + \nu (\nu-1)
f_{\gamma-2,\nu-2} (x) ;\end{align*}
and $f'''_{\gamma,\nu}(x) = O ( x^{\gamma-3} \log^\nu x)$.
Thus Taylor's formula implies that
   \begin{align}
   \label{eq20}
 & ~~  \left( f_{\gamma,\nu} (X_{t} + \Delta_t ) - f_{\gamma, \nu} (X_t) \right)\1_{E_\eps (t)} \nonumber\\
 &
   = \Delta_t \1_{E_\eps(t)} f_{\gamma-1,\nu-1} (X_t) ( \gamma \log (\re+X_t) + \nu ) \nonumber\\
& ~~ {}  + \frac{\Delta_t^2 \1_{E_\eps (t)}}{2} f_{\gamma-2,\nu-2} ( X_t)
 \left( \gamma (\gamma -1) \log^2 (\re+X_t) + \nu (2\gamma -1) \log (\re+X_t) + \nu (\nu -1) \right)   \nonumber\\
& ~~  {} + O ( | \Delta_t|^3 \1_{E_\eps (t)} X_t^{\gamma -3} \log^\nu X_t ) ,\end{align}
as $X_t \to \infty$.
   Since $| \Delta_t |\1_{E_\eps (t)} = O ( X_t^{1-\eps} )$, here we have that
   \[ \Exp [ | \Delta_t|^3 \1_{E_\eps (t)} X_t^{\gamma -3} \log^\nu X_t \mid \F_t ]
   \leq \Exp[ | \Delta_t |^2  \mid \F_t ] O ( X_t^{\gamma -2 -\eps} \log^\nu X_t ) = O ( X_t^{\gamma -2 - ( \eps/2)} ) , \as, \]
     by (\ref{mu2}). On the other hand, since (\ref{moms}) holds for some $p>2$,
   we have from the $q \in \{1,2\}$ cases of Lemma \ref{deltail} that, for $\eps >0$ small enough, a.s.,
   \begin{align*} \Exp [ \Delta_t \1_{E_\eps (t)} \mid \F_t ] & = \Exp [ \Delta_t \mid \F_t] + O ( X_t^{-1-\eps} ) ; \\
   \Exp [ \Delta_t^2 \1_{E_\eps (t)} \mid \F_t ] & = \Exp [ \Delta_t^2 \mid \F_t] + O ( X_t^{-\eps} ) .\end{align*}
Taking expectations in (\ref{eq20}) and using (\ref{mu1}) and (\ref{mu2})
   we obtain, for $\eps>0$ small enough,
\begin{align}
\label{eq29} & ~~ \Exp [  \left( f_{\gamma,\nu} (X_{t} + \Delta_t ) - f_{\gamma, \nu} (X_t) \right) \1_{E_\eps (t)}  \mid \F_t ] \nonumber\\
& = \gamma \left( c + \frac{(\gamma-1)s^2}{2} \right) X_t^{\gamma -2} \log^\nu X_t
+ \nu \left( c + \frac{(2\gamma -1) s^2}{2} + o(1) \right) X_t^{\gamma-2} \log ^{\nu -1} X_t ,\end{align}
as $X_t \to \infty$.
On the other hand, for any $\eps' >0$
there exists $C<\infty$  for which
\[ \left| f_{\gamma,\nu} (X_{t} + \Delta_t ) - f_{\gamma, \nu} (X_t) \right|
\leq C (1 + X_t )^{\gamma + \eps'} + C   | \Delta_t |  ^{\gamma + \eps'} .\]
Hence
$\Exp [  \left| f_{\gamma,\nu} (X_{t} + \Delta_t ) - f_{\gamma, \nu} (X_t) \right| \1_{E^{\rm c}_\eps(t)}  \mid \F_t ]$
is bounded above by \[
  C ( 1 + X_t )^{\gamma + \eps'} \Pr [ E_\eps^{\rm c} (t)   \mid \F_t]
+ C \Exp [ | \Delta_t |  ^{\gamma + \eps'} \1_{E^{\rm c}_\eps(t)}   \mid \F_t ] .\]
For $\eps>0$ small enough,
both terms on the right-hand side here are $O( X_t^{\gamma+\eps'-2-\eps})$,
by the $q=0$ and $q=\gamma +\eps'$ cases of Lemma \ref{deltail} respectively,
the latter case being applicable provided (\ref{moms}) holds for $p > \gamma$ and taking $\eps' \in (0,p-\gamma)$.
 Taking $\eps'$ small enough ($\eps' < \eps/2$, say) and
 combining this last estimate with (\ref{eq29}) we obtain, as $X_t \to \infty$,
\begin{align}
 \label{semieq}
 & ~~ \Exp [    f_{\gamma,\nu} (X_{t} + \Delta_t ) - f_{\gamma, \nu} (X_t)   \mid \F_t ] \nonumber\\
& = \gamma \left( c + \frac{(\gamma-1)s^2}{2} \right) X_t^{\gamma -2} \log^\nu X_t
+ \nu \left( c + \frac{(2\gamma -1) s^2}{2} + o(1) \right) X_t^{\gamma-2} \log ^{\nu -1} X_t ,\end{align}
provided (\ref{moms}) holds for $p > \gamma$.
With the choice $\gamma = 1 +r = 1 - (2c/s^2)$,  (\ref{semieq}) implies (\ref{semieq4})
since $( c + (2 \gamma -1)(s^2/2) ) = (1+r) s^2 /2$ for this choice of $\gamma$.
 Since $r > -1$ and $s^2 >0$, the right-hand side of (\ref{semieq4}) has  the same sign as $\nu$,
 for all $X_t$ large enough, and the conclusion of the lemma follows.
   \end{proof}

\subsection{Technical lemmas}

 We need some results on sums and maxima of i.i.d.\ random variables; first, maxima.

\begin{lm}
\label{lem7}
Let $\zeta_1, \zeta_2, \ldots$ be i.i.d.\ $\R$-valued random variables.
\begin{itemize}
\item[(i)] Suppose that, for some $\theta \in (0,\infty)$ and $\phi \in \R$,
\begin{equation}
\label{iid1}
 \limsup_{x \to \infty} (  x^\theta (\log x)^{-\phi}  \Pr [ \zeta_1 \geq x ] ) < \infty .\end{equation}
For any $\eps >0$,
a.s., for all but finitely many $n$,
$ \max_{1 \leq i \leq n} \zeta_i \leq n^{\frac 1\theta} (\log n)^{\frac{\phi+1}{\theta} +\eps}$.
\item[(ii)]
Suppose that,  for some $\theta \in (0,\infty)$ and $\phi \in \R$,
\begin{equation}
\label{iid2}
 \liminf_{x \to \infty} (  x^\theta (\log x)^{-\phi}  \Pr [ \zeta_1 \geq x ] ) > 0 .\end{equation}
For any $\eps >0$,
a.s., for all but finitely many $n$,
$\max_{1 \leq i \leq n} \zeta_i \geq n^{\frac 1\theta} (\log n)^{\frac{\phi-1}{\theta} -\eps}$.
\end{itemize}
\end{lm}
\begin{proof} First we prove part (i).
From (\ref{iid1}), for some $C \in (0, \infty)$ and all $x$ large enough,
\[ \Pr \Big[ \max_{1 \leq i \leq n} \zeta_i \leq x \Big]
= \prod_{i=1}^n \Pr \left[  \zeta_i \leq x \right]
\geq \left( 1 - C x^{-\theta} (\log x)^\phi \right)^n .\]
Set $x = n^{1/\theta} (\log n)^q$ for some $q \in \R$. Then, for $C' \in (0,\infty)$,
\begin{align*}
p(n) := \Pr \Big[ \max_{1 \leq i \leq n} \zeta_i \geq n^{1/\theta} (\log n)^q \Big]
& \leq 1 - \left( 1 - C' n^{-1} (\log n)^{\phi -\theta q} (1+ o(1)) \right)^n \\
& = O ( 1 \wedge (\log n)^{\phi - \theta q } ) .
\end{align*}
Take $q > (\phi+1)/\theta$. Then $\sum_{k \in \N} p (2^k) < \infty$.
Hence the Borel--Cantelli lemma implies that a.s., for all but finitely many $k \in \N$,
$\max_{1 \leq i \leq 2^k} \zeta_i \leq ( 2^k)^{1/\theta} (\log 2^k )^q$.
For any $n \geq 2$, $2^{k_n} \leq n < 2^{k_n +1}$ for some $k_n \in \N$; hence,
a.s., for all but finitely many $n \in \N$,
\begin{align*}
\max_{1 \leq i \leq  n} \zeta_i & \leq \max_{1 \leq i \leq 2^{k_n + 1} } \zeta_i
\leq (2^{k_n +1} )^{1/\theta} (\log 2^{k_n +1} )^q
  \leq C n^{1/\theta} (\log n)^q ,\end{align*}
where $C<\infty$ does not depend on $n$. Thus we obtain part (i).

Now we prove part (ii). We have from (\ref{iid2}) that for some $c>0$ and all $x$ large enough, $\Pr [ \zeta_1 \geq x]
\geq c x^{-\theta} (\log x)^\phi$, so that
$\Pr  \left[ \max_{1 \leq i \leq n} \zeta_i < x  \right] \leq \left( 1 - c x^{-\theta} (\log x)^\phi \right)^n$.
Taking $x = n^{1/\theta} (\log n)^q$ we obtain
\begin{align*}
\Pr \Big[ \max_{1 \leq i \leq n} \zeta_i < n^{1/\theta} (\log n)^q \Big] & \leq \left( 1- cn^{-1} (\log n)^{\phi - \theta q} (1+o(1)) \right)^n \\
& = O \left( \exp \left( - c (\log n)^{\phi - \theta q} ( 1 + o(1)) \right) \right),
\end{align*}
which is summable over $n \geq 2$ if $q < (\phi -1)/\theta$; now use  the Borel--Cantelli lemma. \end{proof}

 The next result deals with sums of i.i.d.\ nonnegative random variables.

\begin{lm}
\label{mz} Let $\zeta_1, \zeta_2, \ldots$ be
 i.i.d.\ $[0,\infty)$-valued
random
variables.
\begin{itemize}
\item[(i)] If for some $\theta \in (0,1)$ and $\phi \in \R$, (\ref{iid1}) holds,
 then, for any $\eps>0$,  a.s., for all but finitely many $n$,
$\sum_{i=1}^n \zeta_i \leq n^{\frac 1\theta} (\log n)^{\frac{\phi+1}{\theta} +\eps}$.
\item[(ii)]  If, for some $\theta \in (0,\infty)$ and $\phi \in \R$, (\ref{iid2}) holds,
then, for any $\eps >0$,
a.s., for all but finitely many $n$,
$ \sum_{i=1}^n \zeta_i \geq n^{\frac 1\theta} (\log n)^{\frac{\phi-1}{\theta} - \eps }$.
\end{itemize}
\end{lm}
\begin{proof}
Part (i) is a part of a family of classical results related to the
  Marcienkiewicz--Zygmund    strong laws of large numbers (see e.g.\
 \cite[p.\ 73]{kall}): it follows from
a result of Feller \cite[Theorem 2]{feller1} (see also   \cite[p.\ 253]{loeve1} for a more general result).
Part (ii) is a consequence of Lemma \ref{lem7}(ii)
and the elementary bound $\sum_{i=1}^n \zeta_i \geq \max_{1 \leq i \leq n} \zeta_i$. \end{proof}

  Next we move on to some basic consequences of (A0) and (A1).
    Here `i.o.'\ and `f.o.'\ stand for `infinitely often' and `finitely often', respectively.

   \begin{lm}
   \label{irrlem}
   Suppose that (A0) and (A1) hold. Let $R, S \subset \SS$ be finite and non-empty.
   Then $\{ X_t \in R \textrm{ i.o.} \} = \{ X_t \in S \textrm{ i.o.} \}$ up to sets of probability $0$.
   Moreover, for any (hence every) finite, non-empty $R \subset \SS$, the following equalities hold up to sets of probability $0$:
    \begin{align*}
\{ X_t \in R \textrm{ i.o.} \} & = \big\{ \liminf_{t \to \infty} X_t = 0, \, \limsup_{t \to \infty} X_t = \infty \big\}, \\
\{ X_t \in R \textrm{ f.o.} \} & = \big\{ \lim_{t \to \infty} X_t  = \infty \big\}.
\end{align*}
    In particular,
 $\Pr \left[ \{ X_t \to \infty \} \cup \big\{  \liminf_{t \to \infty} X_t = 0, \, \limsup_{t \to \infty} X_t = \infty \big\} \right] = 1$.
   \end{lm}
   \begin{proof} Let $R, S \subset \SS$ be finite and non-empty.
   Suppose that $X_t \in R$ i.o.
   Then, since $R$ is finite,
    there exist $x \in R$, $y \in S$ and stopping times $t_1 < t_2 < \cdots$ with $t_{i+1} > t_i + m(x,y)$
   such that $X_{t_i} = x$ and $\Pr [ X_{t_i + m(x,y)} = y \mid \F_{t_i} ] \geq \varphi (x,y) >0$ for all $i$, by (\ref{irred}).
  Then L\'evy's extension of the Borel--Cantelli lemma (see e.g.\ \cite[p.\ 131]{kall}) implies
   that
   $X_t = y$ i.o., a.s., giving the first statement in the lemma.
   Hence, a.s., either $X_t \in R$ i.o.\ for {\em all} finite non-empty $R \subset  \SS$ (including $R=\{0\}$), or for {\em none}.
   It follows  that  $\liminf_{t \to \infty} X_t \in \{ 0, \infty \}$ a.s.,
  and the same fact also implies that
$\limsup_{t \to \infty} X_t = \infty$ a.s.
     \end{proof}

The next result says, roughly speaking, that uniformly for sites $x$ in some interval, there is positive
probability that, starting from that interval, the process hits $x$ before leaving some
larger interval. We use the notation $\SS_x := \SS \cap [0,x]$ for $x \geq 0$,
\[  \tau_{x,t} := \min \{ s \geq 0 : X_{t+s} = x\}, ~\textrm{and}~
  \sigma_{x,t} := \min \{ s \geq 0 : X_{t+s} > x \}. \]

\begin{lm}
\label{hit0}
Suppose that (A0) and (A1) hold and that for some $C< \infty$,
$\Exp [ \Delta_t \mid \F_t ] \leq C$, a.s.,   for all $t \in \N$.
Let $A < \infty$.
There exist  $\varphi = \varphi(A) >0$ and
$B = B(A, C) > A$ such that for any $x \in \SS_A$, on $\{ X_t \leq A \}$,
$\Pr [ \tau_{x,t} < \sigma_{B,t} \mid \F_t ] \geq \varphi$, a.s., for all $t \in \N$.
\end{lm}
\begin{proof}
From (\ref{irred}), writing
  $m = \max_{x,y \in \SS_A}   m (x,y)  $ and
$\varphi = \min_{x, y \in \SS_A}  \varphi(x,y) $
we have by (A0) and (A1) that $m  < \infty$ and $\varphi > 0$ (depending   on
$A$), and, moreover, on $\{ X_t \leq A \}$, for any $x \in \SS_A$,
$\Pr [ \tau_{x,t} \leq m \mid \F_t ] \geq \varphi$, a.s.
In addition, by an appropriate maximal inequality \cite[Lemma 3.1]{mvw} and the first moment bound in the lemma, on $\{ X_t \leq A \}$,
\[ \Pr [ \sigma_{hm,t} \leq m \mid \F_t ] = \Pr \Big[ \max_{0 \leq s \leq m } X_{t+s} > h m \mid \F_t \Big] \leq \frac{C m + A}{hm} \leq \frac{\varphi}{2}  ,\]
choosing $h$ sufficiently large (depending  on $A$ and $C$). Combining the  two probability bounds
we obtain the statement in the lemma, after a relabelling of  $\varphi/2$ as $\varphi$. \end{proof}

Recall that $\tau_n$ is the time of the $n$th return to $0$ by $X$, and recall that $N$, as defined just before (A3),
is the first $n$ for which $\tau_n = \infty$.

\begin{lm}
\label{yhit}
Suppose that (A0), (A1),
 and (A3) hold, and $N = \infty$ a.s. Then for any $y \in \SS \setminus \{ 0\}$, there
exists $c(y)>0$ such that, for any $n$,
\[ \Pr \left[ (X_t)_{t \geq \tau_n} \textrm{visits } y \textrm{ before time } \tau_{n+1} \mid \F_{\tau_n} \right] = c(y), \as\]
\end{lm}
\begin{proof}
The irreducibility assumption (\ref{irred}) implies that for any $n$, on $\{ \tau_n < \infty \}$,
\begin{equation}
\label{h1}
\Pr [ X_{\tau_n + m(0,y)} = y \mid \F_{\tau_n} ] \geq \varphi (0,y) , \as \end{equation}
By the regenerative assumption (A3), $\Pr \left[ \textrm{hit }  y  \textrm{ before returning to } 0 \mid \F_{\tau_n} \right]$, on $\{ \tau_n < \infty \}$,
does not depend on $n$; call this probability $c(y)$. Then, since $N = \infty$ a.s.,
\[ \Pr [   \textrm{eventually hit }  y ] = \Pr \Big[ \bigcup_{i=1}^\infty \{ \textrm{hit } y  \textrm{ between }   \tau_i  \textrm{ and } \tau_{i+1} \}
\Big] \leq \sum_{i=1}^\infty c(y) .\]
Thus if $c(y) =0$, the probability of eventually hitting $y$ is also $0$, which contradicts (\ref{h1})
(cf Lemma \ref{irrlem}). Hence $c(y)>0$.
\end{proof}

A recurring technical component of our proofs will be controlling the process $X$ in finite intervals such as $[0,x]$,
and the exits (and overshoots) of $X$ from such intervals. The following two lemmas
give basic results in this direction.

\begin{lm}
\label{sigmalem}
Suppose that (A0) and (A1) hold.
For any $x \geq 0$, there exists $\eps>0$
such that, for all $t$ and for all $s$ sufficiently large,
$\Pr [ \sigma_{x,t} > s \mid \F_t] \leq \re^{-\eps s}$ a.s.
In particular, there exists $K < \infty$,
depending on $x$, for which $\Exp [ \sigma_{x,t} \mid \F_t ] \leq K$  a.s.\ for all $t$.
\end{lm}
\begin{proof}
Let $x \geq 0$ and $z \in \SS$, $z>x$.
By (A0) and (A1),
 taking $m = \max_{y \in \SS_x} m(y,z)$ and $\delta =\min_{y \in \SS_x} \varphi (y,z)$
we have $m \in \N$ and   $\delta>0$, depending on $x$, such that, for any $s \geq t$,
\[ \Pr [ \sigma_{x,t} \leq  (s-t) + m \mid \F_s ] \geq \delta \1 \{ \sigma_{x,t} > s-t \} + \1 \{ \sigma_{x,t} \leq s-t \} , \as \]
Taking $s = t + rm$ for $r \in \N$ yields $\Pr [ \sigma_{x,t} > (r+1)m \mid \F_{t+rm} ]
\leq (1-\delta) \1 \{ \sigma_{x,t} > rm \}$, and   a telescoping conditioning argument
at times $t, t+m,  \ldots, t+rm$
gives
 $\Pr [ \sigma_{x,t} > rm \mid \F_t ] \leq (1-\delta)^r$.
For any $s \geq 0$, there is some $r=r(s)$ for which $r m \leq s \leq (r+1)m$, so
\[ \Pr [ \sigma_{x,t} > s \mid \F_t ] \leq \Pr [ \sigma_{x,t} > r m \mid \F_t ] \leq (1-\delta)^r \leq (1-\delta)^{(s/m)-1} ,\]
which implies the result, recalling that $m$ and $\delta$ depend on $x$ but not on $s$ or $t$.
\end{proof}

\begin{lm}
\label{lem55}
Suppose that (A0)--(A3) hold, $r > -1$, and (\ref{moms}) holds with $p > \max \{2, 1+r \}$.
Let $x \geq 0$.
Then for any $\nu \in \R$, there exists $K < \infty$ (depending on $x$) such that, on $\{ X_t \leq x  \}$,
$\Exp [ f_{1+r, \nu} (X_{t+\sigma_{x,t}} ) \mid \F_t ] \leq K$, a.s.
\end{lm}
\begin{proof}
Under the stated conditions,   Lemma \ref{semilem} applies.
In particular,   (\ref{semieq4}) shows that   for any $\eps>0$ there is $C< \infty$, not depending on $x$,
such that for   $s \geq t$,
\[ \Exp [ f_{1+r,\nu}(X_{(s+1) \wedge ( t +\sigma_{x,t} )} ) - f_{1+r,\nu}(X_{s \wedge (t+\sigma_{x,t})} ) \mid \F_s ]
\leq C (1+ X_{s})^{r-1+\eps}   \1 \{ s -t < \sigma_{x,t} \} .\]
Suppose that $X_t \leq x$.
For $t \leq s < t + \sigma_{x,t}$, $X_s \in [0, x ]$, so
writing $b(x) = C \max_{y \in \SS_x}   (1+ y)^{r-1+\eps}   < \infty$, conditioning on $\F_t$
and taking expectations we obtain, on $\{ X_t \leq x\}$, a.s.,
 \[ \Exp [ f_{1+r,\nu}(X_{(s+1) \wedge (t+\sigma_{x,t}) } ) \mid \F_t ] - \Exp [ f_{1+r,\nu}(X_{s \wedge (t+\sigma_{x,t})} ) \mid \F_t ]
\leq b(x)     \Pr [ \sigma_{x,t} > s -t  \mid \F_t ] .\]
Let $u > t$ be an integer. Summing from $s= t$ to $u-1$ we have, on $\{ X_t \leq x \}$, a.s.,
\begin{align*} \Exp [ f_{1+r,\nu}(X_{u \wedge (t+\sigma_{x,t})} )\mid \F_t ] & \leq \Exp [ f_{1+r,\nu}(X_{t} ) \mid \F_t ] + b(x) \sum_{s=0}^\infty
\Pr [ \sigma_{x,t} > s \mid \F_t ] \\
& \leq a(x) + b(x) \Exp[ \sigma_{x,t}
\mid \F_t ]
 ,\end{align*}
writing $a (x) = \max_{y \in \SS_x} f_{1+r,\nu}(y) < \infty$.
The final part of Lemma \ref{sigmalem} then shows that there is $K < \infty$, depending on $x$, for which, for
all $u > t$, $\Exp [ f_{1+r,\nu}(X_{u \wedge (t+\sigma_{x,t})} )\mid \F_t ] \leq K$ a.s.; letting $u \to \infty$, Fatou's lemma completes the proof.
 \end{proof}

\subsection{Proofs of main results from Section \ref{sec:results}}
\label{sec:proofsthms}

 \begin{proof}[Proof of Proposition \ref{dichot}.]
    The first statement of the proposition
     follows from Lemma \ref{irrlem}.
   Now from part (b) of (A3) with a repeated conditioning argument,
   \begin{align*} & {}~~ \Pr [ N > k+1 ] \\
   & = \Pr[ \eta_{k+1} < \infty, \eta_k < \infty,\ldots, \eta_1 < \infty ]  \\
&   = \Pr [ \eta_{k+1} < \infty \mid \tau_k < \infty ] \Pr [ \eta_k < \infty \mid \tau_{k-1} < \infty]
   \cdots \Pr [ \eta_2 < \infty \mid \eta_1 < \infty ] \Pr [ \eta_1 < \infty ] \\ &
   = ( \Pr [\eta_1 < \infty ] )^k.\end{align*}
   If $\Pr [ \eta_1 < \infty ] <1$, this implies that $N < \infty$ a.s.,
    so that $X_t = 0$ f.o., and   Lemma \ref{irrlem}
   shows  that $X_t \to \infty$.
   On the other hand, if $\Pr [ \eta_1 < \infty ]  = 1$ we have that $\Pr [ N > k ] =1$ for any $k$,
   so  $N = \infty$ a.s.\ and hence $X_t =0$ i.o., i.e., $\liminf_{t \to \infty} X_t = 0$, a.s., as claimed.
      \end{proof}

   We  now sketch the proof of Theorem \ref{recthm}.

\begin{proof}[Proof of Theorem \ref{recthm}.] Under slightly different conditions,
this result follows from results of
   \cite{lamp1,lamp3,mai}.
   The results in \cite{lamp1} apply to a more general class
   of processes than we consider here, with a slightly stronger version of (\ref{moms}), while \cite{lamp3} and \cite{mai} state their results in the Markovian setting,
   although their methods  work (as in \cite{lamp1}) in the more general setting; concretely, one can use our Lemma \ref{semilem} (and a variant thereof for $|r|=1$, provided
   by calculations similar to those in \cite{mai})
   together with the results from \cite{lamp1} or \cite{aim}, for instance.
   These papers use a slightly different definition of recurrence
   to ours, but Lemma \ref{irrlem} shows that the definitions are equivalent
   under (A1).
   \end{proof}

 Next we give the proof of Theorem \ref{lem4} on the tail of $M_1 = \max_{1 \leq s \leq \eta_1} X_s$.

    \begin{proof}[Proof of Theorem \ref{lem4}.]
    Throughout the proof fix $r >-1$.
First we prove the lower   bound in (\ref{mtails}). Fix $\nu >0$.
We ease notation by writing $f$ for $f_{1+r,\nu}$ as defined in Section \ref{sec:lyap};
 for $r>-1$ and $\nu>0$, $f$ is nondecreasing on $[0,\infty)$ and $f(z) \to \infty$ as $z \to \infty$.
Lemma \ref{semilem} implies that $ f (X_t)$ satisfies a local submartingale property;   to achieve uniform integrability, we   work
with a truncated version of $f$, namely
$h_x (z) := \min \{ f  (z) , f  ( 2 x) \}$, for fixed $x >0$.
For any $\eps \in (0,1)$, on $\{ X_t \leq x\}$, for all $x$ sufficiently large, $E_\eps (t)$
implies that $X_{t+1} < 2x$. Hence, for any $\eps \in (0,1)$, on $\{ X_t \leq x \}$,
\[ h_x (X_{t+1} ) - h_x (X_t) \geq \left( f  (X_{t+1} ) - f  (X_t) \right) \1_{E_\eps (t) }
- f  (X_t) \1 \{ \Delta_t > (1+ X_t)^{1-\eps} \} ,\]
so that
\begin{align*}
\Exp [ h_x (X_{t+1} ) - h_x (X_t) \mid \F_t ]
& \geq \Exp \left[ \left( f  (X_{t+1} ) - f  (X_t) \right) \1_{E_\eps (t) } \mid \F_t \right]
 - f  (X_t) \Pr [ E_\eps^{\rm c} (t) \mid \F_t ] .\end{align*}
By Lemma \ref{deltail} with  $q=0$, for $\eps>0$ small enough,
$f  (X_t) \Pr [ E_\eps^{\rm c} (t) \mid \F_t ]  = O (X_t^{r-1-(\eps/2)} )$, a.s.;
 with the $\gamma = 1+r$ case of (\ref{eq29}) this shows, as in the proof of
Lemma \ref{semilem}, on $\{ X_t \leq x \}$,
\[ \Exp [ h_x (X_{t+1} ) - h_x (X_t) \mid \F_t ]
\geq ( \nu ( 1+ r) (s^2 /2) + o(1) ) X_t^{r-1} \log^{\nu-1} X_t ,\]
which is positive for all $X_t$ sufficiently large, since $\nu >0$, $r > -1$, and $s^2 >0$.
Thus there exists   $A \in (0,\infty)$  such that, for all $x > A$,
 \begin{equation}
 \label{hx}
 \Exp[ h_x (X_{t+1} )-
h_x (X_t) \mid \F_t ] \geq 0 , ~ \textrm{on} ~ \{ A \leq X_t \leq x \}, \as \end{equation}
Choose $\lambda \in (0,\infty)$ with  $\lambda > \max_{z \in \SS_A} h_x(z)$.
Since $f(y) \to \infty$ as $y \to \infty$,
we can (and do) choose $y \in \SS \cap (A,\infty)$ such that
  $f(y ) > 2 \lambda $. Take $x>y$.
Define the stopping times
\begin{align*}
\kappa_1 & := \min \{ t \in \N : X_t =y \}, \\
\kappa_2 & := \min \{ t > \kappa_1 : X_t \geq x \}, \\
\kappa_3 & := \min \{ t > \kappa_1 : X_t \leq A  \} .\end{align*}
By Lemma \ref{irrlem} and the fact that for $r \geq -1$, $X$ is recurrent (by Theorem \ref{recthm}), $\kappa_i < \infty$ a.s., for each $i \in \{1,2,3\}$.
We consider $(h_x ( X_{t \wedge \kappa_2 \wedge \kappa_3}) )_{t \geq \kappa_1}$, which is a submartingale by (\ref{hx}).
Also, $(h_x ( X_{t \wedge \kappa_2 \wedge \kappa_3}) )_{t \geq \kappa_1}$ is uniformly
integrable (since it is bounded above by $f (2x) < \infty$), so
$h_x(X_{t \wedge \kappa_2 \wedge \kappa_3})$ converges a.s.\ and in $L^1$ to  $h_x(X_{\kappa_2 \wedge \kappa_3})$
as $t \to \infty$. Hence,
\begin{align*} 2 \lambda & \leq \Exp [ h_x ( X_{\kappa_2 \wedge \kappa_3} ) \mid \F_{\kappa_1} ]  \\
& = \Exp [ h_x ( X_{\kappa_2  } ) \1 \{ \kappa_2 < \kappa_3 \} \mid \F_{\kappa_1} ]
+ \Exp [ h_x ( X_{\kappa_3  } ) \1 \{ \kappa_3 < \kappa_2 \} \mid \F_{\kappa_1} ] \\
& \leq  f  (2x) \Pr [ \kappa_2 < \kappa_3 \mid \F_{\kappa_1} ]
+ \lambda \Pr [ \kappa_3 < \kappa_2 \mid \F_{\kappa_1} ] ,\end{align*}
since $h_x (X_{\kappa_3}) \leq f (X_{\kappa_3}) \leq \lambda$  and $h_x ( X_{\kappa_2} ) \leq f  (2x)$ a.s.; re-arranging   we obtain
\begin{equation}
\label{w1}
 \Pr [ \kappa_2 < \kappa_3 \mid \F_{\kappa_1} ]
   \geq \frac{\lambda}{f  (2x) - \lambda} , \as  \end{equation}
  Hence, by (\ref{w1}),
$\Pr [ \kappa_2 < \kappa_3 \mid \F_{\kappa_1} ] \geq 1/ f  (3x)$ for all $x$ large enough. Finally,
\[ \Pr [ \kappa_2 < \eta_1 ] \geq \Exp \left[ \1 \{ \kappa_1 < \eta_1 \} \Pr[ \kappa_2 < \kappa_3 \mid \F_{\kappa_1} ] \right]
\geq \frac{\Pr [ \kappa_1 < \eta_1 ]}{f  (3x)} \geq \frac{1}{f (4x)} ,\]
for all $x$ sufficiently large, by Lemma \ref{yhit}.
On $\{ \kappa_2 < \eta_1 \}$, we have $M_1 = \max_{1 \leq s \leq \eta_1} X_s \geq x$.
Thus we obtain
the lower bound in (\ref{mtails}), since $\nu>0$ was arbitrary.

We now prove the upper bound in (\ref{mtails}).
Fix $\eps>0$ and now write $f$ for $f_{1+r,-\eps}$.  By Lemma
\ref{semilem}, there is $A \in (0, \infty)$ such that
 $\Exp[ f (X_{t+1} )-
f(X_t) \mid \F_t ] \leq 0$ on $\{ X_t  \geq A\}$, a.s.  Since $r>-1$, there
exists $x_0 \geq A$ such that
  $f$ is increasing on $[x_0,\infty)$.
Take $x > x_1 > x_0$; $x_1$ will be fixed  later.
 Define stopping times recursively by $\beta_0 :=1$ and for $n \in \N$,
\begin{align*} \alpha_n & := \min \{ t > \beta_{n-1} : X_t > x_1 \} ,\\
\beta_n & := \min \{ t > \alpha_n : X_t \leq x_0 \}, \\
\gamma_n & := \min \{ t > \alpha_n : X_t > x \} .\end{align*}
By Lemma \ref{irrlem} and the fact that $X$ is recurrent, $\alpha_n$, $\beta_n$ and $\gamma_n$ are a.s.\ finite
for all $n$.

By Lemma \ref{semilem}, $(f(X_{t \wedge \beta_n \wedge \gamma_n}))_{t \geq \alpha_n}$ is a nonnegative
supermartingale,
and as $t \to \infty$ it converges a.s.\ to $f(X_{\beta_n \wedge \gamma_n})$.
By Fatou's lemma, $\Exp [ f (X_{\beta_n \wedge \gamma_n }) \mid \F_{\alpha_n} ] \leq f (X_{\alpha_n})$.
Moreover, $ \Exp [ f (X_{\beta_n \wedge \gamma_n }) \mid \F_{\alpha_n} ] \geq \Pr [ \gamma_n < \beta_n \mid \F_{\alpha_n} ]
f(x)$, since $x > x_0$.
It follows that for all $x >x_1$,
\begin{equation}
\label{eq80}
 \Pr [ \gamma_n < \beta_n \mid \F_{\alpha_n} ] \leq \frac{f (X_{\alpha_n})}{f(x)}. \end{equation}
Lemma \ref{lem55} shows that
$\Exp [ f(X_{\alpha_n} ) ] \leq K$ for some $K< \infty$ depending on $x_1$ but not on $x$.
Thus taking expectations in (\ref{eq80}) we obtain, for some $K <\infty$ and all $x > x_0$,
\[  \Pr [ \gamma_n < \beta_n ] \leq K/ f(x) .\]
Moreover, it follows from Lemma \ref{hit0} that we may choose $x_1 > x_0$ large enough such that,
for some $\delta>0$,
$\Pr [ \eta_1 < \alpha_{n+1} \mid \F_{\beta_n} ] \geq \delta$, a.s.,
for all $n$.
Thus we fix such an $x_1$. In particular, we then have that
$\Pr [ \eta_1 < \beta_{n+1} \mid \F_{\beta_n} ] \geq \delta$, a.s.  Let $J := \min \{ n \in \N : \eta_1 < \beta_n  \}$.
Then $J$ is stochastically dominated by a geometric random variable with parameter $\delta$, and in particular
$\Pr [ J > n ] \leq (1-\delta)^n \leq  \re^{-\delta n}$. Hence
\[ \Pr [ M_1 \geq x ] \leq \Pr \Big[ \bigcup_{n=1}^J \{ \gamma_n < \beta_n \} \Big] \leq \Pr [ J \geq \lfloor k \log x\rfloor ]
+ \sum_{n=1}^{k \log x} \Pr [ \gamma_n < \beta_n ] \leq \frac{C \log x}{f(x)} ,\]
for some $C< \infty$ and all $x > x_1$,
choosing $k$ large enough.
 \end{proof}

We can now state a result on the `maximum of the maxima' in the first $n$ excursions.

\begin{lm}
\label{lem6}
 Suppose that (A0)--(A3) hold.
   Suppose that $r > -1$ and  (\ref{moms}) holds with $p > \max \{ 2, 1+r \}$.
For any $\eps >0$,
a.s., for all but finitely many $n$,
\[ n^{\frac{1}{1+r}} (\log n)^{-\frac{1}{1+r} - \eps} \leq \max_{1 \leq i \leq n} M_i \leq  n^{\frac{1}{1+r}} (\log n)^{\frac{2}{1+r} + \eps} .\]
\end{lm}
\begin{proof}  Apply the tail bounds in Theorem \ref{lem4}
together with Lemma \ref{lem7}. \end{proof}

A key result for several of our remaining theorems is Lemma \ref{lowbound}.
It provides lower bounds on excursion functionals, and in particular gives a new approach to a lower tail bound
for $\eta_1$, which has  advantages  over previous approaches: the results
in \cite{aim} are not as sharp, while the results in \cite{ai1} require  uniformly bounded increments
for the process.

\begin{lm}
\label{lowbound}
Suppose that (A0) and (A1) hold, and there exists $B < \infty$ such that
$\Exp [ \Delta_t^2 \mid \F_t ] \leq B$, a.s., for all $t \in \N$, and there exist $x_0,c \in (0,\infty)$ such that, on $\{ X_t \geq x_0 \}$,
$\Exp [ \Delta_t \mid \F_t ] \geq - c/X_t$, a.s., for all $t \in \N$.
  Let $\Phi : \SS \to [0,\infty)$ be a nondecreasing function.
 Then there exists $\eps >0$ such that for all $y$ sufficiently large,
\[ \Pr [ M_1 \geq y ] \leq 2 \Pr \left[ \sum_{t=1}^{\eta_1} \Phi  (X_t) \geq \eps y^2 \Phi ( y/2) \right] .\]
In particular, $\Pr [  \eta_1 \geq x ] \geq \frac{1}{2} \Pr [ M_1 \geq (x/\eps)^{1/2} ]$ for all $x$ sufficiently large.
\end{lm}
\begin{proof}
Let $y > 2x_0$. Define stopping times
 \[ \kappa_1 = \min \{ t \in \N : X_t \geq y \} ; ~~~ \kappa_2 = \min \{ t \geq \kappa_1 : X_t \leq y/2 \} .\]
 Note that $\kappa_1 < \infty$ a.s., by Lemma \ref{irrlem}, and
 $\{ \kappa _1 < \eta_1 \}$,   the event that
 $X$ reaches $[y,\infty)$ before returning to $0$, is $\F_{\kappa_1}$ measurable.
 Then, for any $\eps >0$,
 \begin{align}
 \label{eq5} \Pr \left[ \{ \kappa_1 < \eta_1 \} \cap \{ \kappa_2 \geq \kappa_1 + \eps y^2 \} \right]
 = \Exp \left[ \1 \{ \kappa_1 < \eta_1 \} \Pr [ \kappa_2 \geq \kappa_1 + \eps y^2 \mid \F_{\kappa_1} ] \right] .\end{align}
 We claim that we may choose $\eps >0$ for which, for all $y$ sufficiently large,
 \begin{equation}
 \label{claim1} \Pr [ \kappa_2 \geq \kappa_1 + \eps y^2 \mid \F_{\kappa_1  } ]
 \geq \frac{1}{2} , \as
 \end{equation}
 To verify (\ref{claim1}), let $W_t = (y - X_t )^2 \1 \{ X_t < y \}$.
 Then on $\{ X_t \geq y \}$, $W_{t+1} - W_t \leq  \Delta_t ^2$, so that
 $\Exp [ W_{t+1} - W_t \mid \F_t ] \leq B$, a.s.,
 on $\{X_t \geq y\}$.
 On the other hand,
 on $\{ X_t < y \}$,
 \begin{align*} W_{t+1} - W_t & \leq (W_{t+1} - W_t ) \1 \{ X_{t+1} < y \} \\
 & \leq -2 (y - X_t )  \Delta_t \1 \{ \Delta_t < y - X_t \} + \Delta_t^2.\end{align*}
 Here we have that
 \begin{align*} -2 (y - X_t )  \Delta_t \1 \{ \Delta_t < y - X_t \}
 & = - 2 (y-X_t) \Delta_t + 2 (y-X_t) \Delta_t \1 \{ \Delta_t > y - X_t \} \\
 & \leq - 2 (y-X_t) \Delta_t + 2 \Delta_t^2 .\end{align*}
 Taking expectations we see that, on $\{ x_0 \leq X_t < y\}$, a.s.,
 \begin{align*}
   \Exp [ W_{t+1} - W_t \mid \F_t ] \leq -2(y -X_t ) \Exp [ \Delta_t \mid \F_t ] + 3 \Exp [ \Delta_t^2 \mid \F_t ] \leq \frac{2 c  y}{X_t} + 3B.
   \end{align*}
   In particular, there exists  $C <\infty$ such that, on $\{ X_t > y/2 \}$,
$\Exp [ W_{ t+1 } - W_{t} \mid \F_t ] \leq C$, a.s. Hence we conclude that, for any $t\geq 0$,
$\Exp [ W_{(\kappa_1 + t+1)\wedge \kappa_2} - W_{(\kappa_1+t) \wedge \kappa_2 } \mid \F_{\kappa_1 +t } ] \leq C$, a.s.
   Then an appropriate maximal inequality (Lemma 3.1 of \cite{mvw}) implies that
   \[ \Pr \Big[ \max_{0 \leq s \leq t} W_{(\kappa_1 +s) \wedge \kappa_2} \geq w \mid \F_{\kappa_1} \Big] \leq C t/w , \as, \]
   using the fact that $W_{\kappa_1} = 0$.
But  $W_{\kappa_2} = (y - X_{\kappa_2} )^2 \geq y^2 /4$, so
 \[ \Pr [ \kappa_2 \leq \kappa_1 + t \mid \F_{\kappa_1} ]
 \leq \Pr \Big[ \max_{0 \leq s \leq t} W_{(\kappa_1 + s) \wedge \kappa_2} \geq (y^2/4) \mid \F_{\kappa_1} \Big]
\leq    \frac{4Ct}{y^2  } , \as,  \]
and choosing $t=\eps y^2$ for $\eps>0$ sufficiently small (not depending on $y$), the claim (\ref{claim1}) follows.
Combining (\ref{eq5}) and (\ref{claim1}) we get
\[  \Pr \left[ \{ \kappa_1 < \eta_1 \} \cap \{ \kappa_2 \geq \kappa_1 + \eps y^2 \} \right]
\geq \frac{1}{2} \Pr [ \kappa_1 < \eta_1 ] = \frac{1}{2} \Pr [ M_1 \geq y ] .\]
On $\{ \kappa_1 < \eta_1 \} \cap \{ \kappa_2 \geq \kappa_1 + \eps y^2 \}$,
$X_s \geq y/2$ for all $\kappa_1 \leq s < \kappa_2$, of which there are
at least $\eps y^2$ values, all before time $\eta_1$; since $\Phi$ is nondecreasing
we obtain the result. \end{proof}

To obtain an upper tail bound on $\eta_1$, one of
the technical ingredients that we need is the following
 consequence of
 Theorem $2'$ of \cite{ai2},
 which extended results in \cite{aim}.

 \begin{lm}
 \label{tail2}
  Suppose that $(Y_t)_{t \in \N}$ is an $(\F_t)_{t \in \N}$-adapted
 stochastic process on an unbounded subset of $[0,\infty)$.
 Let $T_A := \min \{ t \in \N : Y_t \leq A \}$.
 Suppose that there exist $p>0$, $\nu \in \R$,  and $\delta>0$ such that
 \begin{equation}
 \label{tail2eq} \Exp [ Y_{t+1}^{2p}  \log^\nu Y_{t+1}   - Y_{t}^{2p}  \log^\nu Y_{t}    \mid \F_t ]
 \leq -\delta Y_t^{2p-2}  \log^{\nu -1} Y_t ,   ~\textrm{on}~ \{ T_A > t\}. \end{equation}
 Then  for some $C < \infty$,
$\Pr [ T_A \geq x \mid Y_1 = x_0 ] \leq C x^{-p}  ( \log   x )^{p-\nu} x_0^{2p} (\log x_0)^\nu$.
 \end{lm}
 \begin{proof} We apply Theorem $2'$ of \cite{ai2}
 with, in the notation there (but with time denoted by $t$ rather than $n$)
 $X_t = Y_t$,
 $h(x) = x^{2p} (\log x)^\nu$, $U_t = h(Y_t)$, $g(x)= x^{2p-2} (\log x)^{\nu-1}$,
 and $f(x) = x^p (\log x)^q$, where $p>0$ and $\nu \in\R$ are as in the statement of the lemma and $q \in \R$
 is to be chosen later. It follows from Theorem $2'$
 that, under the conditions of the lemma, for some $C <\infty$, for any $x_0 > A$,
 \begin{equation}
 \label{fTA}
 \Exp [ f(T_A) \mid Y_1 = x_0 ] \leq C h(x_0) = C x_0^{2p} (\log x_0)^\nu ,
 \end{equation}
  provided that, writing $f^{-1}$ here for the inverse function of $f$,
 \begin{equation}
 \label{eq3}
  \liminf_{x\to \infty} \left( \frac{g(x)}{f' ( f^{-1} ( h(x)))} \right) > 0.\end{equation}
 We verify (\ref{eq3})
 for the stated $f,g,h$. We claim  there
 exists $c \in (0,\infty)$ such that
 \begin{equation}
 \label{claim2}
 f^{-1} (x) = (c + o(1) ) x^{1/p} (\log x)^{-q/p} .\end{equation}
 Since $f$ is eventually
 increasing, to verify (\ref{claim2})
  it suffices
 to show that, for an appropriate $c \in (0,\infty)$, $f ( (c +\eps) x^{1/p} (\log x)^{-q/p} )$ is eventually greater than $x$
 if $\eps>0$ but eventually less than $x$ if $\eps<0$. But we have, for $\alpha >0$,
 \begin{align*} f( \alpha x^{1/p} (\log x)^{-q/p} )
 & = \alpha^p x (\log x)^{-q} \left[ \log \alpha + p^{-1} \log x - (q/p) \log \log x \right]^q \\
& = (\alpha^p p^{-q} +o(1)) x ,\end{align*}
 which satisfies the desired properties with $\alpha = c$ provided $c^p p^{-q} = 1$, i.e.,
 $c = p^{q/p}$.
  Thus
 we obtain (\ref{claim2}). It follows that, for some $c' \in (0,\infty)$,
 \begin{align*}
  f^{-1} ( h(x))   =    (c'+o(1)) x^2 (\log x)^{\frac{\nu-q}{p}} .\end{align*}
  Now $f'(x) = (p +o(1)) x^{p-1} (\log x)^q$, so we get, for some $c'' \in (0,\infty)$,
  \[ f' ( f^{-1} (h(x))) = ( c'' + o(1) ) x^{2p-2} (\log x)^{\nu + \frac{q-\nu}{p}} .\]
  Then (\ref{eq3}) holds provided that
 $s-1 -  ( \nu + \frac{q-\nu}{p}  ) \geq 0$,
  that is, $q \leq \nu-p$.
 This shows that (\ref{fTA}) holds, and then the result of the lemma follows
 by Markov's inequality.
  \end{proof}

Now we can complete the proof of our main result on the duration of an excursion.

  \begin{proof}[Proof of Theorem \ref{etalem}.]
   First we prove the upper bound in (\ref{etatails}).
   Our starting point will be Lemma
   \ref{tail2}, which deals with the hitting time of a suitably
   large interval $[0,A]$ starting from outside that interval. Some additional work, based on the irreducibility
   assumption, is needed to relate this to the return time $\eta_1$ to $0$.
   Fix $A >0$ and $B >A$ (to be specified later).
      Define stopping times $\alpha_i$ and $\beta_i$ recursively
   by $\beta_0 :=1$ and for $n \in \N$,
   \[
   \alpha_n := \min \{ t \geq \beta_{n-1} : X_t \leq A \} , ~~~
   \beta_n := \min \{ t \geq \alpha_n : X_t \geq B \}
      ;\]
  by Lemma \ref{irrlem} and the fact that $X$ is recurrent, $\alpha_n \leq \beta_n < \infty$
  for all $n$, a.s.

    We have from the $1+r = 2p$ case of  (\ref{semieq})
   that with $Y_t = X_t$, $f_{2p,\nu} (X_t)$ satisfies (\ref{tail2eq})
  taking $2p =1+r >0$ and $\nu<0$, provided the $A$ in (\ref{tail2eq})
  is large enough. Thus take $A$ to be sufficiently large. Hence we can apply Lemma \ref{tail2}
  to show that, for any $\eps>0$,
  \[ \Pr[ \alpha_{n+1} - \beta_n \geq x  \mid \F_{\beta_n}] \leq
  x^{-\frac{1+r}{2}} (\log x)^{\frac{1+r}{2} + \eps} ( 1 + X_{\beta_n} )^{1+r}, \]
  for all $x$ large enough. Here the $\nu = 0$ case of Lemma \ref{lem55}
  shows that
  $\Exp [ ( 1 + X_{\beta_n} )^{1+r} ] \leq K < \infty$ for $K$ not depending on $x$.
So taking expectations in the last display,
  we obtain
  \begin{equation}
  \label{tb1}
   \Pr[ \alpha_{n+1} - \beta_n \geq x ]
   \leq
  x^{-\frac{1+r}{2}} (\log x)^{\frac{1+r}{2} + \eps} ,\end{equation}
  for all $x$ sufficiently large. On the other hand, for $B= B(A)$ as in Lemma \ref{hit0},
  we have that for $\varphi >0$, for all $n$,
  $\Pr [ \eta_1 < \beta_n  \mid \F_{\alpha_n} ] \geq \varphi$, a.s.
      Let $K := \min \{ n : \beta_n > \eta_1 \}$. Then
  $K$ is stochastically dominated by a geometric random variable with
  parameter $\varphi$, and in particular $\Pr [ K > n ] \leq (1-\varphi)^n \leq \re^{-\varphi n}$.
 Moreover, $\eta_1 \leq \sum_{n=1}^K (\alpha_{n+1} -\alpha_n)$. So
 \begin{align}
 \label{ss1}
  \Pr [ \eta_1 \geq x ] & \leq \Pr [ K \geq \lfloor k \log x \rfloor ]
 + \Pr \bigg[ \sum_{n=1}^{k \log x} (\alpha_{n+1} -\alpha_n ) \geq x \bigg]  \nonumber\\
 & \leq x^{-(1+r)} + k ( \log x ) \sup_n \Pr \bigg[ \alpha_{n+1} -\alpha_n \geq \frac{x}{k \log x} \bigg] ,\end{align}
 choosing $k$ sufficiently large. A similar argument to Lemma \ref{sigmalem}
 shows that $\Pr [ \beta_n - \alpha_n \geq x ] \leq \re^{-cx}$ for $c>0$ depending
  on $B$ (and hence on $A$). Then since
  \[ \Pr [ \alpha_{n+1} - \alpha_n \geq x ] \leq \Pr [ \alpha_{n+1} - \beta_n \geq x/ 2 ]
  + \Pr [   \beta_n -\alpha_n \geq x/ 2 ] ,\]
  it follows that $\alpha_{n+1} - \alpha_n$
  satisfies the same tail bound (\ref{tb1}) as $\alpha_{n+1} - \beta_n$.
  The upper bound in (\ref{etatails}) then follows from (\ref{ss1}).

  The lower bound in (\ref{etatails}) follows from the final statement in Lemma \ref{lowbound}
  together with the lower bound in Theorem \ref{lem4}.
 \end{proof}

 Now we can give a result on the total duration of the first $n$ excursions.

 \begin{lm}
 \label{etabounds} Suppose that (A0), (A1), (A2), and (A3) hold.
 \begin{itemize}
\item[(i)] Suppose that $-1 < r \leq 1$ and   (\ref{moms}) holds
with $p > 2$. Then for any $\eps >0$, a.s., for all but finitely many $n$,
\[ n^{\frac{2}{1+r}} (\log n)^{-\frac{2}{1+r} -\eps}  \leq \sum_{i=1}^n \eta_i \leq
n^{\frac{2}{1+r}} (\log n)^{\frac{6+2r}{1+r} +\eps}.\]
 \item[(ii)] Suppose that $r>1$ and   (\ref{moms}) holds
 with $p > 1+r$. Then as $n \to \infty$, a.s., $n^{-1} \sum_{i=1}^n \eta_i \to \Exp [\eta_1] \in (0,\infty)$.
 \end{itemize}
 \end{lm}
 \begin{proof}
 Part (ii) follows from the  strong law of large numbers since
   $\Exp [ \eta_1] < \infty$ for $r >1$, by Theorem \ref{etalem}, while $\Exp [ \eta_1] \neq 0$
 since $\eta_1$ is nondegenerate.
Now
 suppose that  $r \in (-1,1]$. The lower bound in   (i) follows from
the lower bound in (\ref{etatails}) with Lemma
\ref{mz}(ii), while the upper bound in   (i) follows from the upper
bound in (\ref{etatails}) with Lemma \ref{mz}(i). \end{proof}

An inversion of the previous result enables us to complete the proof of our theorem on the number of excursions.
Recall that $N_t = \max \left\{ n \in \N : \sum_{i=1}^n \eta_i \leq t \right\}$.

\begin{proof}[Proof of Theorem \ref{numberthm}.]  For part (i),
fix $\eps>0$. From the lower bound in Lemma \ref{etabounds}(i), we may choose $\eps'>0$ small enough for which,
 a.s., for all $t$ large enough
\[ \sum_{i=1}^{\lceil t^{\frac{1+r}{2}} (\log t)^{1+\eps} \rceil} \eta_i
\geq t (\log t)^{\frac{2\eps}{1+r} -  \eps'} > t ,\]
  giving the upper bound in (\ref{nbounds}).
The lower bound in (\ref{nbounds}) follows similarly from the upper bound in Lemma \ref{etabounds}(i).
Part (ii) follows from Lemma \ref{etabounds}(ii).
\end{proof}

Next we turn to our results on stationary distributions.

\begin{proof}[Proof of Theorem \ref{lemstat}.]
We verify the claimed properties of $\pi$   defined at (\ref{pidef}).
 When $r>1$, we have from Theorem \ref{etalem}
  that $\Exp [\eta_1] \in (0, \infty)$. Since, for any $x$,
 $0 \leq \ell_1 (x) \leq \eta_1$ a.s., it follows that
 $\Exp [ \ell_1 (x) ] < \infty$ for all $x \in \SS$.
 It is clear that $\pi(x) \geq 0$ and $\sum_{x \in \SS} \pi(x) =1$.
To show that $\pi(x) >0$, it suffices to show that $\Exp [ \ell_1 (x) ] >0$.
Suppose that, for some $x \in \SS$,
 $\ell_1 (x) =0$ a.s.  Then by (A3), $L_t(x) =0$ a.s.\ for all $t$. But this contradicts
Lemma \ref{irrlem}. So $\Pr [ \ell_1 (x) >0 ] >0$, which implies $\Exp [ \ell_1 (x) ] >0$.

  Next, note that for any $x \in \SS$, a.s.,
$\sum_{n=1}^{N_t} \ell_n (x) \leq L_t (x) \leq \sum_{n=1}^{N_t+1} \ell_n (x)$.
 Here $(\ell_n (x))_{n \in \N}$ are i.i.d.\ random variables with
 finite means, and so it follows from the strong law of large numbers that
$N_t^{-1} L_t (x) \to \Exp [ \ell_1 (x)]$ a.s.\ for $N_t \to \infty$, which with Theorem \ref{numberthm}(ii)
implies that $t^{-1} L_t(x) \to \frac{\Exp [ \ell_1 (x)]}{ \Exp [ \eta_1]}$ a.s.,
and the $L^q$ convergence follows from the bounded convergence theorem.
Finally the convergence of $\Pr [ X_t = x]$ to $\pi (x)$ follows from
e.g.\ \cite[Corollary VI.1.5, p.\ 171]{asmussen} under the additional `aperiodicity' condition.
\end{proof}

The proofs of our remaining theorems now involve  combining our previous results.

\begin{proof}[Proof of Theorem \ref{xbounds}.] We have that for any $t \in \N$,
\begin{equation}
\label{maxeq} \max_{1 \leq i \leq N_t} M_i \leq \max_{1 \leq s \leq t } X_s \leq \max_{1 \leq i \leq N_t +1 } M_i .\end{equation}
The result follows from Theorem \ref{numberthm} and Lemma \ref{lem6} together with (\ref{maxeq}). \end{proof}

 \begin{proof}[Proof of Theorem \ref{xilem}.] Fix $\alpha \geq 0$ and  $r> -1$.
 First we prove the upper bound in (\ref{xitails}). Clearly $\xi^{(\alpha)}_1 \leq \eta_1 M^\alpha_1$.
It follows that, for any $x>1$,
\begin{align*} \Pr [ \xi^{(\alpha)}_1 \geq x] & \leq \Pr [ \{ \eta_1 \geq x^{\frac{2}{\alpha+2}} (\log x)^{\frac{2\alpha}{\alpha+2}} \}
\cup \{ M_1 \geq x^{\frac{1}{\alpha+2}} (\log x)^{-\frac{2}{\alpha+2}} \} ]  \\
& \leq  \Pr [  \eta_1 \geq x^{\frac{2}{\alpha+2}} (\log x)^{\frac{2\alpha}{\alpha+2}} ] + \Pr [ M_1 \geq x^{\frac{1}{\alpha+2}} (\log x)^{-\frac{2}{\alpha+2}} ] .\end{align*}
Now applying the upper bounds from (\ref{etatails}) and (\ref{mtails}) we obtain the desired
upper bound.

 Next we prove the lower bound in (\ref{xitails}).
  It follows from the $\Phi(x) = x^\alpha$ case of Lemma \ref{lowbound} that there exists $C \in (0,\infty)$ such that,
    for all $x$ large enough,
 \begin{equation}
 \label{claim3}
 \Pr [ \xi^{(\alpha)}_1 \geq x ] \geq \frac{1}{2} \Pr [ M_1 \geq C x^{\frac{1}{\alpha+2}} ] .\end{equation}
 The lower bound in (\ref{xitails}) now follows
  from    (\ref{claim3}) and    the lower bound in (\ref{mtails}).
 \end{proof}

Recall that $S^{(\alpha)}_t =  \sum_{s=1}^t  X_s^\alpha$, so
$S^{(\alpha)}_{\tau_n} = \sum_{i=1}^n  \xi^{(\alpha)}_i$.

 \begin{lm}
 \label{xibounds}
  Suppose that (A0)--(A3) hold. Suppose that $r > -1$ and  (\ref{moms}) holds
with $p > \max \{2,1+r\}$.
 Let $\alpha \geq 0$.
 \begin{itemize}
 \item[ (i)] Suppose that $-1 < r \leq 1 +\alpha$. Then for any $\eps >0$, a.s., for all but finitely many $n$,
\[ n^{\frac{\alpha+2}{1+r}} (\log n)^{-\frac{\alpha+2}{1+r} -\eps}  \leq S^{(\alpha)}_{\tau_n} \leq
n^{\frac{\alpha+2}{1+r}} (\log n)^{\frac{2\alpha+4}{1+r}+2 +\eps}.\]
 \item[(ii)] Suppose that $r > 1+\alpha$. Then as $n \to \infty$, a.s.,
 \begin{equation}
 \label{stau2}
  n^{-1}S^{(\alpha)}_{\tau _n} \to \Exp [\xi^{(\alpha)}_1] = \Exp [ \eta_1] \sum_{x \in \SS} x^\alpha \pi (x) \in (0,\infty), \end{equation}
 where $\pi$ is   given by (\ref{pidef}).
 \end{itemize}
 \end{lm}
 \begin{proof}
First we prove part (ii).
 For $r > 1+\alpha$, $\Exp [ \xi_1^{(\alpha)} ] < \infty$ by Theorem \ref{xilem}.
Then, by (\ref{xidef}),
$\Exp [ \xi^{(\alpha)}_1]    = \sum_{x \in \SS} x^\alpha \Exp [ \ell_1 (x) ]$,
 so, by (\ref{pidef}),
  the two
 expressions for limiting constant in (\ref{stau2}) are indeed equivalent.
 Also, $\Exp [ \xi_1^{(\alpha)} ] >0$ since, by Theorem \ref{lemstat}, $\pi(x)>0$ for all $x \in \SS$.
  The convergence in (\ref{stau2}) follows from the  strong law of large numbers.

 Now for part (i),
 suppose that   $r \in (-1,\alpha+1]$. Then the lower bound in part (i) follows from
the lower bound in (\ref{xitails}) and Lemma
\ref{mz}(ii). The upper bound in part (i) follows from the upper
bound in (\ref{xitails}) and Lemma \ref{mz}(i). \end{proof}

 \begin{proof}[Proof of Theorem \ref{sthm}.]
 By definition of $S^{(\alpha)}_t$ and $N_t$,  for any $t \in \N$,
 \begin{equation}
 \label{eq22} S^{(\alpha)}_{\tau_{N_t}} \leq S^{(\alpha)}_t \leq S^{(\alpha)}_{\tau_{N_t+1}} .\end{equation}

 For $-1<r \leq 1+\alpha$, we have from Lemma \ref{xibounds}(i) with (\ref{eq22})
 that for any $\eps>0$, a.s., for all but finitely many $t$,
 \begin{equation}
 \label{eq23} N_t^{\frac{\alpha+2}{1+r}} (\log t)^{-\frac{\alpha+2}{1+r} -\eps} \leq S_t^{(\alpha)}
 \leq  (N_t+1)^{\frac{\alpha+2}{1+r}}  (\log t)^{\frac{2\alpha+4}{1+r} +2+\eps} ,\end{equation}
  using the fact that $N_t \leq t$ a.s.\ to obtain the logarithmic terms.
 Now from (\ref{eq23}) we obtain part (i) of the theorem by applying
 the bounds for $N_t$ in Theorem \ref{numberthm}(i) and we obtain part (ii)
 of the theorem from Theorem \ref{numberthm}(ii).

 Finally, suppose that $r>1+\alpha$. We have from (\ref{eq22})
 that
 \[ (t^{-1} N_t) N_t^{-1}  S^{(\alpha)}_{\tau_{N_t}} \leq t^{-1} S^{(\alpha)}_t
 \leq (t^{-1} (N_t +1) ) (N_t +1)^{-1}  S^{(\alpha)}_{\tau_{N_t+1}} .\]
 Both $t^{-1} N_t$ and $t^{-1} (N_t+1)$ converge a.s.\
 to $\Exp [\eta_1]^{-1}$ by Theorem \ref{numberthm}(ii), while Lemma \ref{xibounds}(ii)
 and the fact that $N_t \to \infty$ a.s.\ as $t \to \infty$ imply that
 both $N_t^{-1}  S^{(\alpha)}_{\tau_{N_t}}$ and
 $(N_t+1)^{-1}  S^{(\alpha)}_{\tau_{N_t+1}}$ converge a.s.\ to $\Exp [ \xi_1^{(\alpha)}]$.
Hence we obtain the first limit statement in
(\ref{nudef}); for the subsequent equality in (\ref{nudef})  we use
the expression for $\Exp [ \xi_1^{(\alpha)} ]$ given in (\ref{stau2}).
  \end{proof}

\section{Proofs for Section \ref{sec:appl}}
\label{sec:proofs2}

\subsection{Proofs for Section \ref{sec:line}}

In this section we prove our results on processes on the whole real line. The proofs use the same ideas
as those for our results from Section \ref{sec:results}, so we do not dwell on the details.

\begin{proof}[Proof of Theorem \ref{line1}.]
We again use $\tau_0, \tau_1, \ldots$ to denote the times at which $X_t = 0$, and $\eta_n := \tau_n - \tau_{n-1}$. Conditions (B0) and (B1) ensure
that $\eta_1, \eta_2,\ldots$ are i.i.d., and $\eta_1$ has the same distribution as
$\theta_+ \eta_{+} + \theta_- \eta_- + (1-\theta_+ -\theta_-)$, where $\theta_+ := \1 \{ X_2 > 0 \}$, $\theta_- := \1 \{ X_2 < 0\}$ and (by (B2))
$\eta_\pm$ is the return time for a half-line model of Section \ref{sec:results} with $r = r_\pm$, independent of $\theta_+$ and $\theta_-$.
Since $-1 < r_+ < r_- \leq 1$, and $\Exp [ \theta_\pm ] >0$, it follows from
Theorem \ref{etalem} that $\Pr [ \eta_1 \geq x ]  = x^{- \frac{1+r_+}{2} +o(1)}$;
hence the process is null-recurrent.
By a similar argument to Theorem \ref{numberthm}, since each excursion takes either sign with uniformly
positive probability,
 there are $t^{\frac{1+r_+}{2} +o_\omega(1)}$ excursions of each sign by time $t$, where the notation
 $\eps_t = o_\omega (t)$ means that the (random) sequence $\eps_t$ satisfies $\eps_t \to 0$, a.s., as $t \to \infty$
 (in other words, $o_\omega(t)$ is an extension of the Landau $o(1)$ notation in which the implicit constants are allowed
 to depend on the probability space element $\omega$).
The result then follows as in the proof of Theorem \ref{xbounds}, using Lemma \ref{lem6}.
\end{proof}

\begin{proof}[Proof of Theorem \ref{line2}.]
Again, we use the fact that the numbers of positive or negative excursions up until time $t$ are both $t^{\frac{1+r_+}{2} +o_\omega(1)}$, a.s.
Then Lemma \ref{xibounds} and an argument similar to the proof of Theorem \ref{sthm} applied separately to the positive and negative parts $\sum_{s=1}^t X_s^+$ and $\sum_{s=1}^t X_s^-$
shows that, a.s., the latter is $t^{\frac{3}{2} \frac{1+r_+}{1+r_-} + o_\omega(1)}$ while the former is
$t^{\frac{3}{2} + o_\omega(1)}$, which therefore dominates the asymptotics, yielding the result.
\end{proof}

\subsection{Proofs for Section \ref{sec:nonhom}}

We write $\be_1,\ldots, \be_d$ for the standard orthonormal
basis of $\R^d$, and for vectors $\bu, \bv \in \R^d$
we use $\bu \cdot \bv$ to denote
their scalar product.

\begin{proof}[Proof of Theorem \ref{cbrw}.]
Suppose that (C0) holds. Take $\F_t = \sigma (\xi_1, \xi_2, \ldots, \xi_t)$, $X_t = \| \xi_t \|$, and
 $\SS = \{ \| \bx \| : \bx \in \Sigma \} \subset [0,\infty)$. Then $0 \in \SS$ since $\0 \in \Sigma$,
 and, by local finiteness of $\Sigma$,
$\{ \bx \in \Sigma : \| \bx \| = x\}$ is finite for any $x \in \SS$. Hence (A0) follows.
  Next we verify (A1).
By irreducibility of $\Xi$,  for any $\bx, \by \in \Sigma$, there exist $k (\bx, \by) \in \N$
and $\kappa (\bx, \by) >0$ such that
$\Pr [ \xi_{t + k (\bx, \by) } = \by \mid \xi_t = \bx ] = \kappa ( \bx , \by) > 0$.
Let $x= \| \bx \|$ and $y = \| \by \|$, so $x,y \in \SS$.  Then, a.s.,
\begin{align*}
 \Pr [ X_{t+ k (\xi_t , \by)} = y \mid \F_t ] & = \kappa ( \xi_t , \by ) \\
 & \geq \min \{ \kappa ( \bx , \by) : \bx \in \Sigma, \| \bx \| = \| \xi_t \|, \; \by \in \Sigma, \| \by \| = y   \};\end{align*}
denote this last  quantity $\varphi ( \| \xi_t \| , y)$. Then $\varphi ( \| \xi_t \| , y) > 0$   by the finiteness of the
 sets over which $\bx$ and $\by$ run. We  choose $\by$ with $\| \by \| = y$, and for that $\by$ take $m ( \|\xi_t\| , y ) = k (\xi_t, \by ) < \infty$.
This shows that (A1) holds. Moreover, (A3) follows from the fact that $\Xi$ is an irreducible  Markov chain.
Also, by the triangle inequality,
$| X_{t+1} - X_t | = | \| \xi_{t+1} \| - \| \xi_t \| | \leq \| \xi_{t+1} - \xi_t \|$,
so if (\ref{moms2}) holds for some $p>0$, then so does (\ref{moms}).

It remains to show that (C1) implies (A2).
Let $\gamma \in (0,1)$, to be chosen later.
We will  estimate the increment
$\| \xi_t + \theta_t \| - \| \xi_t \|$ by Taylor's theorem in $\R^d$. First observe that
\[ \frac{\partial}{\partial x_i} \| \bx \| = \frac{x_i}{\| \bx \|} ; ~~~ \frac{\partial^2}{\partial x_i \partial x_j} \| \bx \|
= \frac{ \1 \{ i=j \} }{\|\bx\|} - \frac{x_i x_j}{\| \bx \|^3 } ; ~~~ \left| \frac{\partial^3}{\partial x_i \partial x_j \partial x_k} \| \bx \| \right| = O ( \|\bx \|^{-2} ) .\]
Then by Taylor's formula, for any $\bx \in \R^d$,
\begin{align}
\label{eq91}
\left( \| \bx + \theta_t \| - \| \bx \| \right) \1 \{ \| \theta_t \| \leq \| \bx \|^\gamma \}
 & = \sum_{i=1}^d \frac{x_i}{\| \bx \|} ( \theta_t \cdot \be_i )  \1 \{ \| \theta_t \| \leq \| \bx \|^\gamma \} \nonumber\\
 & ~~{}  + \frac{1}{2} \sum_{i=1}^d
\left( \frac{1}{\| \bx\|} - \frac{x_i^2}{\| \bx\|^3} \right) ( \theta_t \cdot \be_i )^2  \1 \{ \| \theta_t \| \leq \| \bx \|^\gamma \} \nonumber\\
& ~~ {}
 - \sum_{i=2}^d \sum_{j=1}^{i-1} \frac{x_i x_j}{\| \bx \|^3} ( \theta_t \cdot \be_i ) ( \theta_t \cdot \be_j ) \1 \{ \| \theta_t \| \leq \| \bx \|^\gamma \} \nonumber\\
 & ~~ {}  + O \left( \| \theta_t \|^3 \| \bx \|^{-2} \1 \{ \| \theta_t \| \leq \| \bx \|^\gamma \} \right) .\end{align}
We will condition on $\xi_t = \bx$ and take expectations in (\ref{eq91}).
To this end, note that
\[ \Exp [ \| \theta_t \|^3  \| \bx \|^{-2} \1 \{ \| \theta_t \| \leq \| \bx \|^\gamma \} \mid \xi_t = \bx]
 \leq \| \bx \|^{\gamma -2}\Exp [ \| \theta_t \|^2  \mid \xi_t = \bx ] = O ( \|\bx \|^{\gamma -2} ) .\]
 In addition, for $q \in [0,2]$ a similar argument to Lemma \ref{deltail} shows that, for some $\eps>0$ and $\gamma$ close enough to $1$,
 \begin{equation}
 \label{eq92}
  \Exp [ (\theta_t \cdot \be_i )^q \1 \{ \| \theta_t \| > \| \bx \|^\gamma \} \mid \xi_t = \bx] = O ( \| \bx \|^{q-2-\eps} ) .\end{equation}
 The $q=2$ case of (\ref{eq92}),  with the Cauchy--Schwarz inequality, shows that
 \[  \Exp [ (\theta_t \cdot \be_i ) (\theta_t \cdot \be_j ) \1 \{ \| \theta_t \| > \| \bx \|^\gamma \} \mid \xi_t = \bx] = O ( \| \bx \|^{-\eps} ) .\]
 Also, we have from (C1) that
 \begin{align*} \Exp [ \theta_t \cdot \be_i \mid \xi_t = \bx ] & = \be_i \cdot \mu (\bx) = \frac{ \rho x_i}{\| \bx \|^2} + o ( \| \bx \|^{-1} \log ^{-1} \| \bx \| ) ; \\
 \Exp[ ( \theta_t \cdot \be_i) ( \theta_t \cdot \be_j) \mid \xi_t = \bx ] & = M_{ i j} (\bx ) = \sigma^2 \1 \{ i = j\} + o (  \log ^{-1} \| \bx \| ) .\end{align*}
 Combining these estimates, taking expectations in (\ref{eq91}) yields
 \begin{align*}
& ~~{} \Exp \left[ \left( \| \bx + \theta_t \| - \| \bx \| \right) \1 \{ \| \theta_t \| \leq \| \bx \|^\gamma \}
 \mid \xi_t = \bx \right] \\
 & = \sum_{i=1}^d \frac{\rho x_i^2}{\| \bx \|^3} + \frac{1}{2} \sum_{i=1}^d \left( \frac{1}{\| \bx\|} - \frac{x_i^2}{\| \bx\|^3} \right) \sigma^2
 + o ( \| \bx \|^{-1} \log^{-1} \| \bx \| ) \\
 & = \left( \rho + \frac{\sigma^2}{2} (d-1) \right) \| \bx \|^{-1}  + o ( \| \bx \|^{-1} \log^{-1} \| \bx \| ).\end{align*}
 On the other hand, by the triangle inequality,
 \[ \Exp \left[ \left| \| \bx + \theta_t \| - \| \bx \| \right| \1 \{ \| \theta_t \| > \| \bx \|^\gamma \}
 \mid \xi_t = \bx \right]   \leq   \Exp \left[ \| \theta_t \|  \1 \{ \| \theta_t \| > \| \bx \|^\gamma \} \right] ,\]
 which, for $\gamma$ close enough to $1$, is also $o ( \| \bx \|^{-1} \log^{-1} \| \bx \| )$ by another application of (\ref{eq92}).
 Thus we have shown that
 \begin{equation}
 \label{eq95}
  \Exp [ X_{t+1} - X_t \mid \xi_t = \bx ] = \left( \rho + \frac{\sigma^2}{2} (d-1) \right) \| \bx \|^{-1}  + o ( \| \bx \|^{-1} \log^{-1} \| \bx \| ) ,\end{equation}
 which implies that (\ref{mu1}) holds with $c = \rho + (d-1) (\sigma^2/2)$.

 For the second moment estimate, observe that, given $\xi_t = \bx$,
 \begin{align}
 \label{eq93} ( X_{t+1} - X_t )^2 & = \| \bx + \theta_t \|^2 - \| \bx \|^2 -2 \| \bx \| ( \| \bx + \theta_t \| - \| \bx \| ) \nonumber\\
 & = \| \theta_t\|^2 + 2 \bx \cdot \theta_t   -2 \| \bx \| ( X_{t+1} - X_t ) .\end{align}
 Here we have that
 \begin{equation}
 \label{eq94} \Exp [ \| \theta_t \|^2 + 2 \bx \cdot \theta_t \mid \xi_t = \bx ] = \sum_{i=1}^d M_{ii} (\bx) + 2 \bx \cdot \mu (\bx) = d \sigma^2 + 2 \rho +
 o ( \log^{-1} \| \bx \| ) .\end{equation}
 Taking expectations in  (\ref{eq93}), using (\ref{eq94}) and (\ref{eq95}), we obtain
 \[ \Exp [ (X_{t+1} - X_t)^2 \mid \xi_t = \bx ] =  d \sigma^2 + 2 \rho - 2 ( \rho + (d-1) (\sigma^2/2) ) +  o ( \log^{-1} \| \bx \| ) ,\]
 which, after simplification, shows that (\ref{mu2}) holds with $s^2 = \sigma^2$.
\end{proof}

\subsection{Proofs for Section \ref{sec:shu}}
\label{shuproofs}

First we prove the following analogue of Lemma 7.6 of \cite{shu}. As in \cite{shu}, we relate the general
version of $Z_t$ to the special case in which $\kappa = 0$ a.s., which we denote here by $Z_t'$.
By construction, for $x, y \in \N$,
\begin{equation}
\label{zz}
 \Pr [ Z_{t+1} = y \mid Z_t = x ] = \Exp \left[ \Pr[ Z'_{t+1} = y \mid Z'_t = x - \min \{ \kappa ,   x-1 \} ] \right] .\end{equation}
Write $D_t := Z_{t+1} - Z_t$.

\begin{lm}
\label{shulem}
For any $\eps>0$, as $x \to \infty$,
\begin{equation}
\label{dtail}
 \Pr [ | D_t | > x^{(1/2)+\eps} \mid Z_t = x] = O ( \exp \{ - x^{\eps /3} \} ) .\end{equation}
Also, for any $r \in \N$ there exists $C < \infty$ for which, for all $x \in \N$,
\[ \Exp [ | D_t |^r \mid Z_t = x ] \leq C x^{r/2} .\]
Moreover,   as $x \to \infty$,
\begin{align}
\Exp [ D_t \mid Z_t = x ] & = \frac{2}{3} - \Exp [ \kappa ] + o ( \log^{-1} x ) , \label{d1} \\
\Exp [ D_t^2 \mid Z_t = x ] & = \frac{2}{3} x    + o( x \log^{-1} x) . \label{d2} \end{align}
 \end{lm}
\begin{proof}
The proof is similar to that of Lemma 7.6 in \cite{shu}; we sketch the differences, which
are due to the fact that we use (\ref{zz}) in place of the final statement of Lemma 7.5 in \cite{shu}.
Write $D_t':= Z_{t+1}'-Z_t'$. Lemma 6.4 in \cite{shu} says that, for a given $\alpha >0$,
\begin{equation}
\label{d0}
\Exp [ D_t' \mid Z_t' = x] = \frac{2}{3} + O (\re^{-\alpha x}).
\end{equation}
We also note that, by Markov's inequality and our tail assumption on $\kappa$, there is $C<\infty$ for which,
for all $r \geq 1$,
\begin{equation}
\label{kap}
\Pr [ | \kappa | > r ] \leq C \re^{-\lambda r} .\end{equation}

We prove (\ref{dtail}). By (\ref{zz}), for any $\eps>0$, and any $x >1$,
\[ \Pr [ | D_t | > r \mid Z_t = x] \leq \Pr [ | \kappa | > x^\eps ] + \sup_{y : | x-y | \leq x^\eps }
 \Pr [ | D'_t | > r - x^\eps \mid Z'_t =x ] .\]
 Taking $r = x^{(1/2)+\eps}$, using (\ref{kap}) and the tail bound for $D_t'$ given in Lemma 6.3 of \cite{shu},
 we verify (\ref{dtail}). For (\ref{d1}),
it follows from (\ref{zz}) that
\[ \Exp [ D_t \mid Z_t = x] = - \Exp [ \min \{ \kappa , x-1\} ] + \Exp \left[ \Exp [ D_t' \mid Z_t' = x - \kappa ] \right] .\]
Here, as in the proof of Lemma 7.6 in \cite{shu}, $\Exp [ \min \{ \kappa , x-1\} ] = \Exp [ \kappa] + O (\exp \{-\lambda x/2\})$.
Also, using the fact that $\sup_x \Exp [ D_t' \mid Z_t' = x] < C < \infty$ by (\ref{d0}), we have
\begin{align*} \Exp \left[ \Exp [ D_t' \mid Z_t' = x - \kappa ] \right]
& \leq C \Pr [ | \kappa | > \sqrt{x} ] + \sup_{y : | y- x | \leq \sqrt{x}} \Exp[ D_t' \mid Z'_t = y ] ,\end{align*}
which with (\ref{d0}) and (\ref{kap})  gives the upper bound in (\ref{d1}), a similar argument yielding the lower bound. Similar
variations of the arguments in the proof of Lemma 7.6 of \cite{shu} give
the remaining parts of the lemma.
\end{proof}

\begin{proof}[Proof of Proposition \ref{shuprop}.]
Proposition \ref{shuprop} follows from Lemma \ref{shulem} in exactly the same way as
Lemma 7.7 in \cite{shu} follows from Lemma 7.6 there.
\end{proof}

\begin{proof}[Proof of Theorem \ref{shuthm}.]
We proceed as in the proof of Theorem 2.6 of \cite{shu}, but instead of Lemma 8.3 in \cite{shu},
we apply our sharper Theorem \ref{xilem}. The details involve minor modifications to
the arguments in \cite{shu}. Here we merely give some intuition as to why $\xi_1^{(2)}$ appears.
The key fact is that, ignoring the jumps driven by $\kappa$,  the original process
 takes $Z_t + Z_{t+1}$ steps to traverse the quadrant between times $\nu_t$ and $\nu_{t+1}$;
 the correction to this due to the jump of size $\kappa_t$ is small. Hence over one excursion
 of the embedded process $X_t = \sqrt{Z_t-1}$, the original process accumulates time
 $\tau \approx \sum_{t=1}^{\tau_q} X_t^2$, which is exactly of the form of the excursion sum $\xi_1^{(2)}$.
\end{proof}

\end{document}